\documentclass[smallextended,referee,envcountsect,]{svjour3}
\smartqed
\usepackage{graphicx}
\journalname{JOTA}
\usepackage{amsmath, amssymb}

\newcommand{\g}{\gamma}

\newcommand{\di}[1]{\,\mathrm{d}#1}

\newtheorem{thm}{Theorem}
\makeatletter

\usepackage[utf8]{inputenc}

\renewcommand\subsection{\@startsection{subsection}{2}{\z@}%
  {-3.25ex\@plus -1ex \@minus -.2ex}%
  {1.5ex \@plus .2ex}%
  {\normalfont\bfseries}}
\makeatother

\usepackage{hyperref}
\makeatletter
\renewcommand\subsection{\@startsection{subsection}{2}{\z@}%
  {-3.25ex\@plus -1ex \@minus -.2ex}%
  {1.5ex \@plus .2ex}%
  {\normalfont\bfseries}}
\makeatother

\makeatletter
\def\makeheadbox{\relax}
\makeatother

\begin{document}

\title{An optimal control problem for Stokes--Cahn--Hilliard--Oono equations with regular potential}
\titlerunning{An optimal control problem for phase field modeling}
\authorrunning{A. Kundu}

\author{Arghya Kundu}

\institute{Arghya Kundu \at
             Indian Institute of Technology Kharagpur \\
             Kharagpur, India\\
              arghyakundu5@gmail.com
}

 \date{}

 \maketitle

\begin{abstract}
This article discusses an optimal control problem for a  phase field model of two immiscible incompressible  fluid flow, incorporating surface tension effects. The optimal control problem is defined with a $L^2$--cost functional and subject to the constraints governed by a system of coupled Stokes--Cahn--Hilliard--Oono equations. In this model, fluids are separated by a dynamic diffuse interface of finite width. We investigate the optimality condition of a given control. Initially, we establish the existence of an optimal solution for the coupled optimal control problem. Subsequently, we derive the optimality condition with respect to the corresponding adjoint system.
\end{abstract}


\section{Introduction}
This article focuses on examining an optimal control problem associated to a phase field method for describing the movement of two incompressible immiscible fluids within a smooth bounded domain $\Omega \subset \mathbb{R}^n\,(n=2,3) $. The proposed optimal control problem: find $\theta \in \mathcal{U}_{ad}$ such that 
 \begin{align}
      J(\theta)=\frac{1}{2}\int_{(0,T)\times  \Omega}  |v-v_d|^2\, \di(x,t)+ \frac{1}{2}\int_{(0,T)\times  \Omega} |u-u_d|^2 \, \di(x,t)\nonumber\\+ \frac{\beta}{2}\int_{(0,T)\times  \Omega}|\theta|^2 \, \di(x,t) \label{1.7}
 \end{align}
    reaches its infimum with respect to $\theta$. Subject to the control constraints
    \begin{equation*}
        \theta \in \mathcal{U}_{ad} := \{ \theta \in L^2 (0,T; \mathcal{H}): \theta_{min} \leq \theta \leq \theta_{max}\,\text{a.e. in }(0,T) \times \Omega\}
    \end{equation*}
    and the state equations
    \begin{subequations}
        \begin{align}
        \partial_t v -\mu\Delta v + \nabla p&= - \lambda u\nabla w+ \theta && \textnormal {in}\quad (0,T)\times \Omega,\label{1.3a}\\
\nabla \cdot v&=0 && \textnormal {in}\quad (0,T)\times \Omega,\label{1.3b}\\
v&=0 && \textnormal {on}\quad (0,T)\times \partial \Omega,\label{1.3c}\\
 v (0,x)&=v_0 (x)&& \text{in}\quad \Omega,\label{1.3d}\\
\partial_t u  +  v \cdot \nabla u+\alpha u& =  \Delta w&& \text{in}\quad (0,T)\times \Omega, \label{1.3e}\\
w &=-\Delta u+f(u) && \text{in}\quad (0,T)\times \Omega,\label{1.3f}\\
\nabla u \cdot \vec{n}&=0 && \text{on} \quad (0,T)\times \partial\Omega,\label{1.3g}\\
\nabla w \cdot \vec{n}&=0 && \text{on}\quad (0,T)\times \partial\Omega,\label{1.3h}\\
u(0,x)&=u_0(x)&& \text{in}\quad \Omega,\label{1.3i}
    \end{align}
    \end{subequations}
where $v_d$, $u_d$ refers to the desired states of $v$ and $u$, respectively; $\alpha >0$, $\beta > 0$ are constants. We define the state equations \eqref{1.3a} -- \eqref{1.3i} by $\mathcal{(P)}$ and the entire control problem \eqref{1.7} -- \eqref{1.3i} by $\mathcal{(OP)}$.

For $\alpha =0$, the system \eqref{1.3a} -- \eqref{1.3i} describes the system of Stokes--Cahn--Hilliard equations. Here, $\mu$ is viscosity,  $ v$ is  unknown  Eulerian velocity and $\lambda$ is interfacial width parameter. 
The order parameter $u$ represents concentration or the difference in volume fraction. It is constrained to attain values -1 and 1 in regions occupied by pure fluids, while its values range between -1 and 1 within the diffuse interface with a width proportional to $\lambda$. The function $w$ represents the chemical potential associated with $( v,u)$. The term $f(u)$ is derived from the derivative of the homogeneous free energy functional $F$, which penalizes deviations from the constraint $|u|\leq 1$. Although there are few other choices of $F$ (logarithmic functional or non smooth functional, see  \cite{blowey1991cahn} \cite{copetti1992numerical}), in this study, $F$ is assumed to be a quadratic double well free energy functional.
\begin{equation*}
    F(u)=\frac{1}{4}(u^2-1)^2.
\end{equation*}
The nonlinear term $\lambda u \nabla w$ in \eqref{1.3a} models the surface tension effect in the incompressible Stokes equations \eqref{1.3a} -- \eqref{1.3b}. Equations \eqref{1.3e} -- \eqref{1.3f}  model the advective Cahn--Hilliard equation with advection effect $ v \cdot \nabla u$ in \eqref{1.3e}. 
\par Note that the advective term in \eqref{1.3e} is changed by using the incompressibility condition $\nabla \cdot  v= 0$, since
	$ \nabla \cdot [ v u] 
	= [\nabla \cdot v] u 
	+ v\cdot \nabla u 
	= v\cdot \nabla u .$ The surface tension term $-\lambda u\nabla w$ in \eqref{1.3a} can be replaced by $\lambda w\nabla u$ from the identity $\nabla  (\lambda u w)=\lambda u \nabla w+ \lambda w \nabla u$, where the additional gradient term is absorbed by into the pressure.
 For more on phase field modeling and their analysis we refer to \cite{anderson1998diffuse}, \cite{kim2012phase}, \cite{santra2020phase}, \cite{nurnberg2017numerical}, \cite{feng2006fully}, \cite{kim2012phase} and references therein. 

The Oono model, a variant of the classical Cahn--Hilliard equation, was proposed in \cite{oono1987computationally} (see Chapter 4 of \cite{villain2010phases}) to incorporate long range, nonlocal interactions and to improve computational efficiency in phase ordering simulations (though no numerical experiments using the Cahn–Hilliard–Oono model are presented here). Analysis of this system provides key insights into processes such as spinodal decomposition, pattern formation, and microstructure evolution in materials. A simple linear term, $\alpha u$ with $\alpha > 0$, is introduced in equation \eqref{1.3e} to represent these long range effects. The resulting Cahn–Hilliard–Oono system takes the form:
 \begin{subequations}
     \begin{align}
         \partial_t u + \alpha u&= \Delta w && \text{in}\,\,\, (0,T)\times  \Omega,\label{1.2a}\\
        w&= -\Delta u+ f(u) && \text{in}\,\,\, (0,T)\times  \Omega. \label{1.2b}
     \end{align}
 \end{subequations}
 One can observe that \eqref{1.2a} and \eqref{1.2b} are obtained by considering the free energy
 \begin{equation}
     E(u)=\frac{1}{2} |\nabla u|^2+F(u)+ \int_{\Omega} u(x')g(x',x)u(x) \ dx',
 \end{equation}
 where $| \cdot |$ denotes the Euclidean norm and  $g(x',x)$ describes the long ranged interactions. In particular, for the Oono model and in three dimensional cases
 \begin{equation}
     g(x',x)=\frac{ \alpha }{4\pi|x'-x|}, \quad \alpha > 0.
 \end{equation}
 It is worth noting that, the long range interactions exhibit a repulsive nature when $u(x)$ and $u(x')$ have opposite signs and hence encourage to form interfaces. Finally, the derivation of the Cahn--Hilliard--Oono equations is
 \begin{equation}
     \frac{\partial u}{\partial t} = \nabla^2_x\Big(\frac{\delta E(u)}{\delta u}\Big)= \nabla^2_x\Big(f(u)-\nabla^2u+\int_\Omega u(x')g(x',x)\,dx'\Big),  
 \end{equation}
 where $\frac{\delta}{\delta u}$ denotes the variational derivative. Note that $- \frac{1}{|x'-x|}$ is the Green function associated with the Laplacian operator. Now 
 \begin{align*}
     \nabla^2_x\Big(\int_\Omega u(x')g(x',x)\,dx'\Big)=\int_\Omega u(x') \nabla^2_x g(x',x)\, dx'=& -\alpha\int_\Omega u(x') \delta(x',x)\,dx'\\
     =&-\alpha u(x),
 \end{align*}
 which gives us the equations \eqref{1.2a} and \eqref{1.2b}. After adding $\alpha u$, system no longer  satisfies the conservation of mass and this makes the problem more interesting to derive the estimates of the order parameter and the time derivative of the order parameter. For more on system of Cahn--Hilliard--Oono model, we refer to \cite{villain2010phases}, \cite{miranville2011asymptotic}, \cite{miranville2016cahn}, \cite{he2022viscous} and references therein.

The optimal control of Cahn–Hilliard–Oono systems has attracted growing attention in recent years, particularly due to its applications in phase separation processes with relaxation effects. In \cite{colli2021well}, the authors have studied the distributed optimal control problem of a Cahn--Hilliard--Oono system admitting general potential that include both the case of a regular potential and the case of some singular potential. In \cite{gilardi2023well}, the authors studied an optimal control problem for a viscous Cahn–Hilliard–Oono system with dynamic boundary conditions. They derived first-order necessary optimality conditions using an adjoint system and addressed the existence of optimal controls within a standard distributed control framework. In \cite{zhang2024optimal}, the author investigated a distributed optimal control problem for the viscous Cahn–Hilliard–Oono system incorporating chemotaxis effects. The analysis included the derivation of optimality conditions and explored how chemotactic sensitivity influences the evolution and control of the phase field. Several works have addressed optimal control problems for phase-field models involving Cahn–Hilliard equations coupling with Navier–Stokes equations. In \cite{medjo2015optimal}, the authors investigate Pontryagin's maximum principle for a class of optimal control problems governed by a coupled Cahn–Hilliard–Navier–Stokes system in a two-dimensional bounded domain. In \cite{frigeri2016optimal},  authors study a distributive optimal control problem associated to a diffuse interface model for incompressible isothermal mixtures of two immiscible fluids coupling the Navier–Stokes system with a convective nonlocal Cahn–Hilliard equation in two dimensions of space. In \cite{hintermuller2017optimal}, the authors address the distributed optimal control of a time-discrete Cahn–Hilliard–Navier–Stokes system with variable densities. In \cite{biswas2020pontryagin}, the authors establish Pontryagin’s maximum principle and derive second-order optimality conditions for optimal control problems governed by two-dimensional nonlocal Cahn–Hilliard–Navier–Stokes equations. In \cite{zhao2023optimal}, the authors study an optimal distributed control problem for a two-dimensional Navier–Stokes–Cahn–Hilliard system incorporating chemotaxis and singular potential. For more on optimal control problems on coupled Cahn--Hilliard--Navier--Stokes equations, interested readers may look into \cite{biswas2020maximum}, \cite{tachim2021maximum}, \cite{dharmatti2021nonlocal}, \cite{hintermuller2024strong} and references therein. Several numerical studies on optimal control problem for this type of coupling equations have been developed, see \cite{hintermuller2018goal}, \cite{garcke2019optimal}, \cite{grassle2019simulation}. There are some contributions associated to optimal control problems for the Navier--Stokes or Cahn--Hilliard system in \cite{colli2015optimal}, \cite{colli2015deep}, \cite{colli2015boundary}, \cite{colli2017optimal}, \cite{fursikov2005optimal},  \cite{garcke2018optimal}, \cite{rocca2015optimal}, \cite{zhao2014optimal}, \cite{zhao2013optimal}. However, the optimal control problem of the Stokes--Cahn--Hilliard--Oono system with regular potential has never been explored in the literature.
\par The remainder of the paper is structured as follows. In Section 2, we present the well-posedness results for the state system $\mathcal{(P)}$
and derive the necessary estimates required for the analysis of the control problem. Section 3 contains the main results of the work, including the proof of existence of an optimal solution for the optimal control problem $\mathcal{(OP)}$, the derivation of the Fréchet differentiability of the control-to-state operator and the formulation of the first-order optimality conditions via the adjoint system. Additionally, the existence of solutions to both the linearized and adjoint systems is established in this section.

\section{Well-posedness of $\mathcal{(P)}$}

Let $r \in \mathbb{N}_0,\ 1\leq m < \infty$, the notations $L^m(\Omega)$ and $H^{r,m}(\Omega)$ conventionally denote the Lebesgue and Sobolev spaces with their respective norms $||\cdot||_{L^m}$ and $||\cdot||_{H^{r,m}}$. 
For a Banach space $V$ and its dual $V^*$, duality pairing is denoted by $ \langle \cdot \ 
,\  \cdot  \rangle_{V^* \times V}$. Symbols $\hookrightarrow$, $\hookrightarrow\hookrightarrow$ and $\underset{\hookrightarrow}{d}$ denote the continuous, compact and dense embeddings, respectively. In the convergence part, $\overset{w}{\rightharpoonup}$, $\overset{w^*}{\rightharpoonup}$ and $\rightarrow$ denote weak, weak$^*$ and strong convergences, respectively.

We define the function spaces:
\begin{align*}
\mathcal{V}& = \{\eta : \eta \in H^1_0 (\Omega)^n, \nabla \cdot \eta = 0 \},\quad \\
\mathcal{H}&= \overline{ \{u \in C^\infty_0(\Omega)^n: \nabla \cdot u=0\}}^{L^2(\Omega)^n}. 
\end{align*}
Let $A=-P\Delta$ be the Stokes operator and $P$ be the Leray projector such that $P : L^2(\Omega)^n \to \mathcal{H}$. Before study the weak formulation and the existence of weak solution parts, we consider the following assumptions:
\begin{enumerate}
    \item[(A1)]  $v_0$ and $ u_0\, \geq 0$ for all $x \in \Omega$.
    \item[(A2)] $ v_0 \in \mathcal{V}$ and $ u_0 \in  H^1(\Omega)$. 
    \item[(A3)] $\mathcal{U}$ is a nonempty bounded open subset of $L^2(0,T; \mathcal{H})$ containing $\mathcal{U}_{ad}$ and there exists $K >0$ such that 
    \begin{equation*}
        || \theta ||_{L^2(0,T; \mathcal{H})} \leq K, \,\,\text{for all } \theta \in \mathcal{U}.
    \end{equation*}
    \item[(A4)] $\beta > 0$, $v_d\in L^2(0,T; \mathcal{V})$, $u_d \in L^2(0,T; \mathcal{V})$, $\theta_{min}, \theta_{max} \in L^\infty((0,T) \times \Omega)$ such that $\theta_{min} \leq \theta_{max}$ for a.e. in $(0,T) \times \Omega$. 
\end{enumerate}

\subsection{Weak solution and basic properties}
\begin{definition}\label{defi 2.1}
    Let the assumptions (A1) -- (A4) hold true, then a triplet $( v, u,w)$ is called a weak solution of $\mathcal{(P)}$, if it satisfies $ v(0,x)=  v_0(x)$ and $u(0,x)=u_0(x)$ for all $x \in \Omega$ and 
 \begin{align}
      \int_S \langle\partial_t  v, \psi \rangle \, dt+  \mu \int_{(0,T)\times  \Omega} \nabla  v : \nabla \psi \, \di(x,t)=-\lambda \int_{(0,T)\times  \Omega} u \nabla w \cdot \psi\,\di(x,t)\nonumber\\
      +\int_{(0,T)\times  \Omega} \theta \cdot \psi\, \di(x,t), \label{2.8}\\
     \int_S \langle\partial_t u, \phi \rangle \, dt+ \int_{(0,T)\times  \Omega} \nabla w \cdot \nabla \phi\, \di(x,t) + \alpha \int_{(0,T)\times  \Omega} u \phi \,\di(x,t)\nonumber\\
     = \int_{(0,T)\times  \Omega}u  v \cdot \nabla \phi \,\di(x,t),\label{2.9}\\
     \int_{(0,T)\times  \Omega}w \varphi \,\di(x,t) =\int_{(0,T)\times  \Omega}\nabla u \cdot \nabla \varphi \, \di(x,t)+ \int_{(0,T)\times  \Omega} f(u)\varphi\, \di(x,t), \label{2.10}
 \end{align}
 for all $\psi \in L^2(0,T;\mathcal{V})$ and $\phi, \varphi \in L^2(0,T; H^1(\Omega))$.
\end{definition}

 \begin{thm}\label{Theorem 3.1}
     Let the assumptions (A1) -- (A4) hold true, then solution $(v, u,w)$ of $\mathcal{(P)}$  in the sense of Definition \ref{defi 2.1} satisfies the following a-priori estimate:
     \begin{align}
  ||u||_{L^\infty(0,T;L^4(\Omega))}+ || \partial_t u||_{L^2(0,T;H^1(\Omega)^
  *)}+ ||\nabla u||
  _{L^\infty(0,T; L^2(\Omega))}+||w||_{L^2(0,T;H^1(\Omega))}\nonumber\\
  +||\nabla w||_{L^2((0,T)\times  \Omega)}+
  ||u||_{L^2(0,T; H^3(\Omega))}+ ||v||_{L^2(0,T; L^2(\Omega)^n)}+||\nabla v||_{L^2((0,T)\times  \Omega)^{n\times n}}\nonumber\\
  + || \partial_t v||_{L^2(0,T;\mathcal{V}^*)} \leq C < \infty, \label{2.11}
     \end{align}
     where $C=C(||v_0||_{L^2(\Omega)^n},||u_0||_{H^1(\Omega)}, || \theta||_{ L^2(0,T; L^2(\Omega)^n)}, T)$ is a constant.
 
 \end{thm}
\begin{proof}
    Step 1: Multiplying \eqref{1.3a}, \eqref{1.3e} and \eqref{1.3f} by $v$, $\lambda w$ and $\lambda \partial_t u$, respectively and integrating the resulting equalities over $\Omega$, we obtain
    \begin{align}
        &\frac{1}{2}\frac{d}{dt} \int_{\Omega} | v|^2 \, dx+  \mu \int_{\Omega} |\nabla  v|^2\, dx +\lambda \int_{\Omega} u \nabla w \cdot  v \, dx= \int_{ \Omega} \theta \cdot v \, dx, \label{2.13}\\
        & \lambda\int_{\Omega}\partial_t u w \,dx + \lambda \int_{\Omega} |\nabla w|^2 \,dx-\lambda\int_{\Omega} u  v \cdot \nabla w \, dx+ \alpha \lambda\int_{\Omega} u w \,dx=0,\label{2.14}\\
        &- \lambda\int_{\Omega} \partial_t u w \, dx+\frac{\lambda}{2} \frac{d}{dt}\int_{\Omega} |\nabla u|^2\, dx + \lambda \frac{d}{dt} \int_{\Omega} F(u) \, dx =0. \label{2.15}
    \end{align}
\end{proof}
Adding \eqref{2.13}, \eqref{2.14} and \eqref{2.15} and integrating over $(0,t)$, we get
\begin{align}
    \frac{1}{2}\int_{\Omega} | v(t)|^2\,dx+  \mu \int_{0}^{t}\int_{\Omega} |\nabla  v|^2\, \di(x,t)+  \lambda \int_{0}^{t}\int_{\Omega} |\nabla w|^2 \, \di(x,t)\nonumber\\
    +\frac{\lambda}{2} \int_{\Omega} |\nabla u(t)|^2 \, dx
    +\lambda\int_{\Omega} F(u(t))\, dx 
    + \alpha \lambda\int_{0}^{t}\int_{\Omega} u w \, \di(x,t)\nonumber\\
    \leq   \frac{1}{2}\int_{\Omega} | v(0)|^2\,dx
    +\frac{\lambda}{2} \int_{\Omega} |\nabla u(0)|^2 \, dx+ \frac{1}{2} \int_{0}^{t} || v(t)||^2_{L^2(\Omega)}\,dt
   \nonumber\\ + \frac{1}{2}\int_{0}^{t} ||\theta(t)||^2_{L^2(\Omega)}\,dt.\label{2.16}
\end{align}
Now,
\begin{align}
    &\alpha \lambda\int_{0}^{t}\int_{\Omega} u w \, \di(x,t) \nonumber\\
    =&\, \alpha \lambda\int_{0}^{t}\int_{\Omega} |\nabla u|^2 \,\di(x,t)+ \alpha \lambda\int_{0}^{t}\int_{\Omega} ((u)^4-(u)^2)\,\di(x,t) \notag\\
     =&\,\alpha \lambda\int_{0}^{t}\int_{\Omega} |\nabla u|^2 \,\di(x,t)+ \alpha \lambda\int_{0}^{t}\int_{\Omega} \Big((u)^2-\frac{1}{2}\Big)^2 \,\di(x,t)- C, \label{2.17}
\end{align}
where the constant $C >0$. 
 Hence, from \eqref{2.16} and \eqref{2.17}, we obtain 
\begin{align}
    \frac{1}{2}\int_{\Omega} | v(t)|^2\,dx+  \mu \int_{0}^{t}\int_{\Omega} |\nabla  v|^2\, \di(x,t)+  \lambda \int_{0}^{t}\int_{\Omega} |\nabla w|^2 \, \di(x,t)\nonumber\\
    +\frac{\lambda}{2} \int_{\Omega} |\nabla u(t)|^2 \, dx
    +\lambda\int_{\Omega} F(u(t))\, dx 
    + \alpha \lambda\int_{0}^{t}\int_{\Omega} |\nabla u|^2 \,\di(x,t)\nonumber\\
    + \alpha \lambda\int_{0}^{t}\int_{\Omega} \Big(u^2-\frac{1}{2}\Big)^2 \,\di(x,t) \leq \frac{1}{2} \int_{0}^{t} || v(t)||^2_{L^2(\Omega)}\,dt + C_1, \label{2.18}
\end{align}
where the constant $C_1=C_1(||v_0||_{L^2(\Omega)^n},||u_0||_{H^1(\Omega)}, || \theta||_{ L^2(0,T; L^2(\Omega)^n)},T)$.\\
\par Hence, Gronwall's inequality implies
\begin{equation*}
    || v(t)||_{L^2(\Omega)} \leq 2C_1e^T < \infty\,\, \text{a.e.}\,\, t \in (0,T).
\end{equation*}
Note that $F(u)=\frac{1}{4}((u)^2-1)^2\,\geq0$, then from inequality \eqref{2.18} $\int_{\Omega} F(u) \,dx \leq C_2$ implies
\begin{align*}
    \int_{\Omega} |u|^4\, dx\leq C_2+ 2\delta \int_{\Omega}|u|^4\, dx+ (C_\delta+1)|\Omega|.
\end{align*}
Choose $\delta=\frac{1}{4}$ in the above inequality then 
\begin{align*}
    ||u||_{L^\infty(0,T;L^4(\Omega))} < \infty.
\end{align*}
Hence, 
\begin{align}
     || v||_{L^2(0,T; L^2(\Omega)^n)}+||\nabla v||_{L^2((0,T)\times  \Omega)^{n \times n}}+||\nabla w||_{L^2((0,T)\times  \Omega)}+ ||\nabla u||_{L^2((0,T)\times  \Omega)}
   \notag\\+||u||_{L^\infty(0,T;L^4(\Omega))} \leq C_3 < \infty . \label{2.19}
\end{align}
Step 2: Now, choose $\psi \in \mathcal{V}$, then
 \begin{align}
     &\,|\langle\partial_t  v, \psi\rangle|\nonumber\\ \leq&  \,\mu || \nabla  v||_{L^2(\Omega)} || \nabla \psi||_{L^2(\Omega)}+||u||_{L^4(\Omega)}||\nabla w||_{L^2(\Omega)} ||\psi||_{L^4(\Omega)}+ ||\theta||_{L^2(\Omega)} || \psi||_{L^2(\Omega)}\notag\\
     \leq&\,  \Big(\mu || \nabla  v||_{L^2(\Omega)} + ||u||_{L^4(\Omega)}||\nabla w||_{L^2(\Omega)}+ ||\theta||_{L^2(\Omega)}\Big) ||\psi||_{\mathcal{V}} \nonumber.
 \end{align}
 Hence, 
 \begin{equation}
     \underset{||\psi||_{\mathcal{V}}\leq 1}{\text{sup}}  |\langle\partial_t  v, \psi\rangle| \leq  \mu || \nabla  v||_{L^2(\Omega)}+||u||_{L^4(\Omega)}||\nabla w||_{L^2(\Omega)} + ||\theta||_{L^2(\Omega)}. \label{2.20}
 \end{equation}
Now, integrating \eqref{2.20} on $(0,T)$ and using the estimate \eqref{2.18}, we get
\begin{equation}
    || \partial_t  v||_{L^2(0,T;\mathcal{V}^*)}\leq C_4 < \infty. \nonumber
\end{equation} 
Similarly,  for $ \phi \in H^1(\Omega)$
\begin{align}
   &\, |\langle\partial_t u, \phi\rangle|\nonumber\\
    \leq&\, ||\nabla w||_{L^2(\Omega)}|| \nabla \phi||_{L^2(\Omega)}+ ||u||_{L^4(\Omega)}||  v||_{L^4(\Omega)}||\nabla \phi||_{L^2(\Omega)}+ \alpha ||u||_{L^2(\Omega)} || \phi||_{L^2(\Omega)}\notag\\
    \leq&\, \Big(||\nabla w||_{L^2(\Omega)}+||u||_{L^4(\Omega)}||  v||_{L^4(\Omega)}+ \alpha ||u||_{L^2(\Omega)}\Big)|| \phi||_{H^1(\Omega)}. \nonumber
\end{align}
Hence,
\begin{align}
     \underset{||\phi||_{H^1(\Omega) }\leq 1}{\text{sup}} &|\langle\partial_t u, \phi\rangle| \leq ||\nabla w||_{L^2(\Omega)}+||u||_{L^4(\Omega)}||  v||_{L^4(\Omega)} \notag\\
    \Rightarrow \,& || \partial_t u||_{L^2(0,T;H^1(\Omega)^
    *)} \leq C_5 <\infty. \nonumber
\end{align}
Integrating \eqref{1.3f} over $\Omega$ to obtain
\begin{equation}
    \int_{\Omega} w \, dx= \int_{\Omega}f(u)\,dx= \int_{\Omega} \{(u)^3-u\} \, dx \leq C_6< \infty. \label{2.21}
    \end{equation}
    Now, Poincaré and triangle inequalities yield 
    \begin{align}
        ||w||_{L^2((0,T)\times  \Omega)} < \infty \Rightarrow ||w||_{L^2(0,T;H^1(\Omega))} \leq C_7< \infty. \nonumber
    \end{align}
Therefore from \eqref{1.3f}, we obtain $||\Delta u||_{L^2((0,T)\times  \Omega)}+ ||\nabla \Delta u||_{L^2((0,T)\times  \Omega)}\leq C_8 < \infty$.
     Hence, by the virtue of the regularity theory of elliptic partial differential equations of second order we obtain $|| u||_{L^2(0,T; H^3(\Omega))}\leq C_9 < \infty$, where all the constants $C_i=C_i(||v_0||_{L^2(\Omega)^n},||u_0||_{H^1(\Omega)}, || \theta||_{ L^2(0,T; L^2(\Omega)^n)}, T)$. An estimate for the pressure term will be provided in the following theorem.
    
\begin{thm}\label{thm 3.2}
        Let the assumptions (A1) -- (A4) hold true, then the problem $\mathcal{(P)}$ possesses at least one weak solution $(v,u,w)$ such that
        \begin{align*}
            &v \in L^2(0,T;\mathcal{V}), \, \partial_t v \in L^2 (0,T;\mathcal{V}^*);\\
           & u \in C ([0,T];H^1(\Omega))\cap L^2(0,T; H^3(\Omega)), \, \partial_t u \in L^2(0,T; H^1(\Omega)^*);\\
           & w \in L^2(0,T;H^1(\Omega)),
        \end{align*}
and $(v,u,w)$ satisfies the weak formulations \eqref{2.8}, \eqref{2.9} and \eqref{2.10}. Furthermore, there exists a pressure term $p:= \partial_t P$ associated to each weak solution $(v, u,w)$, satisfies the \eqref{1.3a} in the distributional sense. The pressure term $P \in L^\infty (0,T;L^2_0(\Omega))$ satisfies 
\begin{align*}
    \underset{t \in [0,T]}{\text{sup}}||\nabla P||_{H^{-1}(\Omega)^n} \leq C < \infty.
\end{align*}
    \end{thm}
    \begin{proof} Step 1:
Let $\{\eta_i\}_{i=1}^\infty$ be a family of eigenfunctions  to the Stokes operator $A$ with the corresponding non-decreasing sequence of eigenvalues $\{\kappa_i\}_{i=1}^\infty$ such that $A\eta_i=\kappa_i \eta_i$.  $\{\kappa_i\}_{i=1}^\infty$ are complete orthonormal basis of $\mathcal{H}$ and orthogonal to $\mathcal{V}$. Further assume that  $\{\xi_i\}$ be the orthonormal base in $L^2(\Omega)$ being composed of the eigenfunctions of the operator $-\Delta$ with the boundary condition \eqref{1.3g}. 
\par Next, take $X_N=$ span $\{\xi_1, \xi_2,\cdots, \xi_N\}$, $Y_N=$ span $\{\eta_1, \eta_2, \cdots, \eta_N\}$. Let $P_{X_N}$ and $P_{Y_N}$ be the orthogonal projectors from $L^2(\Omega)$ to $X_N$ and $\mathcal{H}$ to $Y_N$, respectively. Then we look for functions 
\begin{align*}
    u_N=\displaystyle \sum_{i=1}^N a_{N_i}(t) \xi_i \in X_N,\, w_N=\displaystyle \sum_{i=1}^N b_{N_i}(t) \xi_i \in X_N\,\, \text{and} \,\, v_N=\displaystyle \sum _{i=1}^N c_{N_i}(t)\eta_i \in Y_N
\end{align*}
solve the following approximate problem

\begin{align}
    &  \langle\partial_t  v_N, \xi \rangle +  \mu \int_{\Omega} \nabla  v_N  : \nabla \xi \, dx =-\lambda \int_{ \Omega} u_N \nabla w_N \cdot \xi\,dx + \int_{\Omega} \theta \cdot \xi,\label{eq3.14}\\
     &  \langle\partial_t u_N, \psi \rangle \,+ \int_{\Omega} \nabla w_N \cdot \nabla \psi\, dx + \alpha \int_{ \Omega} u_N \psi \,dx= \int_{ \Omega}u_N  v_N \cdot \nabla \psi \,dx,\label{eq3.15}\\
     &\int_{\Omega}w_N \zeta \,dx =\int_{ \Omega}\nabla u_N \cdot \nabla \zeta \, dx+ \int_{\Omega} f(u_N)\zeta\, dx\label{eq3.16} 
\end{align}
for any $\xi \in Y_N$ and $\psi, \zeta \in X_N$
with $a_{N_i}(0)=  \langle u_0, \xi_i \rangle \,;\,\,   c_{N_i}(0)= \langle  v_0, \eta_i \rangle $ for all $i=1,2,\cdots N$ and $\theta \in L^2(0,T;\mathcal{H})$. Note that $u_{N}(0)= \displaystyle\sum_{n=1}^N \langle u_0, \xi_i \rangle \xi_i $ and $ v_{N}(0)=\displaystyle\sum_{i=1}^N \langle  v_0, \eta_i \rangle \eta_i$. Observe that the approximating problem becomes a Cauchy problem for a system of ordinary differential equations with $2N$ unknowns.


\par Choose $\xi = v_N, \, \psi=w_N$ and $\zeta= \partial_t u_N$, then we obtain
\begin{align}
    \frac{1}{2}\int_{\Omega} | v_N(t)|^2\,dx+  \mu \int_{0}^{t}\int_{\Omega} |\nabla  v_N|^2\, \di(x,t)+  \lambda \int_{0}^{t}\int_{\Omega} |\nabla w_N|^2 \, \di(x,t)\nonumber\\
    +\frac{\lambda}{2} \int_{\Omega} |\nabla u_N(t)|^2 \, dx
    +\lambda\int_{\Omega} F(u_N(t))\, dx 
    + \alpha \lambda\int_{0}^{t}\int_{\Omega} u_N w_N \, \di(x,t)\nonumber\\
    \leq   \frac{1}{2}\int_{\Omega} | v_N(0)|^2\,dx+\frac{\lambda}{2} \int_{\Omega} |\nabla u_N(0)|^2 \, dx+ \frac{1}{2} \int_{0}^{t} || v_N(t)||^2_{L^2(\Omega)}\,dt\nonumber\\
    + \frac{1}{2}\int_{0}^{t} ||\theta(t)||^2_{L^2(\Omega)}\,dt. \label{equ3.14}
\end{align}
Now, 
\begin{align}
   &\, \alpha \lambda\int_{0}^{t}\int_{\Omega} u_N w_N \, \di(x,t) \nonumber\\
   = &\,\alpha \lambda\int_{0}^{t}\int_{\Omega} |\nabla u_N|^2 \,\di(x,t)+ \alpha \lambda\int_{0}^{t}\int_{\Omega} ((u_N)^4-(u_N)^2)\,\di(x,t) \notag\\
    =&\, \alpha \lambda\int_{0}^{t}\int_{\Omega} |\nabla u_N|^2 \,\di(x,t)+ \alpha \lambda\int_{0}^{t}\int_{\Omega} \Big((u_N)^2-\frac{1}{2}\Big)^2 \,\di(x,t)- C, \label{equ3.15}
\end{align}
where the constant $C >0$. 
 Hence from \eqref{equ3.14} and \eqref{equ3.15}, we obtain 
\begin{align}
    \frac{1}{2}\int_{\Omega} | v_N(t)|^2\,dx+  \mu \int_{0}^{t}\int_{\Omega} |\nabla  v_N|^2\, \di(x,t)+  \lambda \int_{0}^{t}\int_{\Omega} |\nabla w_N|^2 \, \di(x,t)\nonumber\\
    +\frac{\lambda}{2} \int_{\Omega} |\nabla u_N(t)|^2 \, dx
    +\lambda\int_{\Omega} F(u_N(t))\, dx 
    + \alpha \lambda\int_{0}^{t}\int_{\Omega} |\nabla u_N|^2 \,\di(x,t)\nonumber\\
    + \alpha \lambda\int_{0}^{t}\int_{\Omega} \Big(u^2_N-\frac{1}{2}\Big)^2 \,\di(x,t) \leq \frac{1}{2} \int_{0}^{t} || v_N(t)||^2_{L^2(\Omega)}\,dt + C_1, \label{equ3.16}
\end{align}
where the constant $C_1 > 0$.
Hence, Gronwall's inequality implies
\begin{equation*}
    || v_N(t)||_{L^2(\Omega)} < \infty\,\, \text{a.e.}\,\, t \in (0,T).
\end{equation*}
Note that $F(u_N)=\frac{1}{4}((u_N)^2-1)^2\,\geq0$, then from inequality \eqref{equ3.16} $\int_{\Omega} F(u_N) \,dx \leq C_2$ implies
\begin{align*}
    \int_{\Omega} |u_N|^4\, dx\leq C_2+ 2\delta \int_{\Omega}|u_N|^4\, dx+ (C_\delta+1)|\Omega|.
\end{align*}
Choose $\delta=\frac{1}{4}$ in the above inequality then 
\begin{align*}
    ||u_N||_{L^\infty(0,T;L^4(\Omega))} < \infty.
\end{align*}
Hence, 
\begin{align}
     || v_N||_{L^2(0,T; L^2(\Omega)^n)}+||\nabla v_N||_{L^2((0,T)\times  \Omega)^{n \times n}}+||\nabla w_N||_{L^\infty(0,T; L^2(\Omega))}\nonumber\\
     + ||\nabla u_N||_{L^2((0,T)\times  \Omega)}
   +||u_N||_{L^\infty(0,T;L^4(\Omega))}  < \infty . \label{equ3.17}
\end{align}
Next choose $\zeta= w_N$, then we get 
\begin{align*}
   &\int_{\Omega} w_N^2\,dx= \int_{\Omega}\nabla u_N \cdot \nabla w_N\, dx+ \int_{\Omega}f(u_N)w_N\,dx\\
    & \hspace{1.5cm} \leq ||\nabla u_N||_{L^2(\Omega)}|| \nabla w_N||_{L^2(\Omega)}+ ||f(u_N)||_{L^2(\Omega)}||w_N||_{L^2(\Omega)}\\
    \Rightarrow &\,\,||w_N||_{L^2((0,T)\times  \Omega)} < \infty.
\end{align*}
Hence, the Poincaré inequality yields
\begin{align*}
    ||w_N||_{L^2(0,T;H^1(\Omega))} < \infty.
\end{align*}
Further choose $\zeta = \Delta u_N$, then
\begin{align*}
&\int_{\Omega} w_N \Delta u_N \, dx= \int_{\Omega} \nabla u_N \cdot \nabla (\Delta u_N)\, dx+ \int_{\Omega}f(u_N) \Delta u_N\,dx\\
\Rightarrow& \int_{\Omega} (\Delta u_N)^2 \, dx \leq ||w_N||_{L^2(\Omega)}||\Delta u_N||_{L^2(\Omega)}+ ||f(u_N)||_{L^2(\Omega)}||\Delta u_N||_{L^2(\Omega)}\\
\Rightarrow& \,\,||\Delta u_N||_{L^2((0,T)\times  \Omega)} < \infty.
\end{align*}
Step 2: Now, choose any $\xi \in \mathcal{V}$, then
 \begin{align}
    &\, |\langle\partial_t  v_N, \xi\rangle|\nonumber\\
    \leq &\, \mu || \nabla  v_N||_{L^2(\Omega)} || \nabla \xi||_{L^2(\Omega)}+||u_N||_{L^4(\Omega)}||\nabla w_N||_{L^2(\Omega)} ||\xi||_{L^4(\Omega)}+ ||\theta||_{L^2(\Omega)} || \xi||_{L^2(\Omega)}\notag\\
     \leq&\,  \Big(\mu || \nabla  v_N||_{L^2(\Omega)} + ||u_N||_{L^4(\Omega)}||\nabla w_N||_{L^2(\Omega)}+ ||\theta||_{L^2(\Omega)}\Big) ||\xi||_{\mathcal{V}} \nonumber.
 \end{align}
 Hence, 
 \begin{equation}
     \underset{||\xi||_{\mathcal{V}}\leq 1}{\text{sup}}  |\langle\partial_t  v_N, \xi\rangle| \leq  \mu || \nabla  v_N||_{L^2(\Omega)}+||u_N||_{L^4(\Omega)}||\nabla w_N||_{L^2(\Omega)} + ||\theta||_{L^2(\Omega)}. \label{equ3.18}
 \end{equation}
Now, integrating \eqref{equ3.18} on $(0,T)$ and using the estimate \eqref{equ3.17}, we get
\begin{equation}
    || \partial_t  v_N||_{L^2(0,T;\mathcal{V}^*)}  < \infty. \nonumber
\end{equation}
 
Similarly,  for any $ \psi \in H^1(\Omega)$
\begin{align}
    &\,|\langle\partial_t u_N, \psi\rangle| \nonumber\\
     \leq&\, ||\nabla w_N||_{L^2(\Omega)}|| \nabla \psi||_{L^2(\Omega)}+ ||u_N||_{L^4(\Omega)}||  v_N||_{L^4(\Omega)}||\nabla \psi||_{L^2(\Omega)}+ \alpha ||u_N||_{L^2(\Omega)} || \psi||_{L^2(\Omega)}\notag\\
     \leq &\, \Big(||\nabla w_N||_{L^2(\Omega)}+||u_N||_{L^4(\Omega)}||  v_N||_{L^4(\Omega)}+ \alpha ||u_N||_{L^2(\Omega)}\Big)|| \psi||_{H^1(\Omega)}. \nonumber
\end{align}
Hence,
\begin{align}
     \underset{||\psi||_{H^1(\Omega) }\leq 1}{\text{sup}} &|\langle\partial_t u_N, \psi\rangle| \leq ||\nabla w_N||_{L^2(\Omega)}+||u_N||_{L^4(\Omega)}||  v_N||_{L^4(\Omega)}+ \alpha ||u_N||_{L^2(\Omega)} \notag\\
    \Rightarrow \,& || \partial_t u_N||_{L^2(0,T;H^1(\Omega)^
    *)} < \infty. \nonumber
\end{align}

    
 Hence, by the virtue of the regularity theory of elliptic partial differential equations of second order, we obtain
    \begin{align*}
        || u_N||_{L^2(0,T; H^3(\Omega))} < \infty.
    \end{align*}
    Hence, we obtain
\begin{align*}
   & \{ v_N\}^\infty_{N=1} \text{  is uniformly bounded in $L^2(0,T; \mathcal{V})$},\\
    &\{ u_N\}^\infty_{N=1} \text{  is uniformly bounded in $L^\infty(0,T;H^1(\Omega)) \cap L^2(0,T; H^3(\Omega))$},\\
    & \{ w_N\}^\infty_{N=1} \text{ is uniformly bounded in $L^2(0,T; H^1(\Omega))$},\\
    &\{ \frac{\partial v_N}{\partial t}\}^\infty_{N=1} \text{  is uniformly bounded in $L^2(0,T;\mathcal{V}^*)$},\\
    &\{ \frac{\partial u_N}{\partial t}\}^\infty_{N=1} \text{  is uniformly bounded in $L^2(0,T;H^1(\Omega)^*)$}.
\end{align*}
Then, there exist
\begin{align*}
    &v \in L^2(0,T; \mathcal{V}),\\
    & u \in L^\infty(0,T;H^1(\Omega)) \cap L^2(0,T; H^3(\Omega)),\\
    & w \in L^2(0,T; H^1(\Omega)),\\
    & \frac{\partial v}{\partial t} \in L^2(0,T;\mathcal{V}^*),\\
    & \frac{\partial u}{\partial t}  \in L^2(0,T;H^1(\Omega)^*)
\end{align*}
and subsequences $\{ v_N\}^\infty_{N=1},\,\{ u_N\}^\infty_{N=1},\, \{ w_N\}^\infty_{N=1},\, \{ \frac{\partial v_N}{\partial t}\}^\infty_{N=1} ,\, \{ \frac{\partial u_N}{\partial t}\}^\infty_{N=1}$ (still indexed by same symbol) such that 
\begin{align*}
    &v_N \overset{w}{\rightharpoonup} v \,\,\text{in}\,\, L^2(0,T; \mathcal{V}),\\
     &u_N \overset{w}{\rightharpoonup} u \,\,\text{in}\,\, L^2(0,T; H^3(\Omega)),\\
     & u_N \overset{w*}{\rightharpoonup} u \,\,\text{in}\,\, L^\infty(0,T;H^1(\Omega)),\\
     & w_N \overset{w}{\rightharpoonup} w \,\,\text{in}\,\, L^2(0,T; H^1(\Omega)),\\
     &  \frac{\partial v_N}{\partial t} \overset{w}{\rightharpoonup} \frac{\partial v}{\partial t}\,\,\text{in}\,\, L^2(0,T;\mathcal{V}^*),\\
     & \frac{\partial u_N}{\partial t} \overset{w}{\rightharpoonup} \frac{\partial u}{\partial t}\,\,\text{in}\,\, L^2(0,T;H^1(\Omega)^*).
\end{align*}
Now, passing these limits in the approximate problem \eqref{eq3.14} -- \eqref{eq3.16}, we get a weak solution. \\
\par Step 2: Now using the proposition III.1.1 of \cite{temam2001navier}, the identity \eqref{2.8} implies the existence of pressure term $p:= \partial_t P \in W^{-1, \infty}(0,T;L^2_0(\Omega))$ such that 
\begin{align*}
    -\int_{\Omega}P(t) \nabla \cdot \psi dx=& - \int_{\Omega} (v(t)-v_0) \cdot \psi \,dx- \mu \int_{0}^{t} \int_{\Omega} \nabla v(s) : \nabla \psi \, \di(x,s)\\
    &-\int_{0}^{t} \int_{\Omega} u \nabla w \cdot \psi \, \di(x,s)- \int_{0}^{t} \int_{\Omega} \theta \cdot \psi \, \di(x,s) \,\, \text{for all}\,\, \psi \in H^1_0(\Omega)^n.
\end{align*}
Therefore, we obtain 
\begin{align*}
   &\, \langle \nabla P(t), \psi \rangle\nonumber\\
   \leq & \,\Big(|| v(t)||_{L^2(\Omega)} - ||v_0||_{L^2(\Omega)}\Big)||\psi||_{L^2(\Omega)}+ \mu \int_{0}^{t} ||\nabla v||_{L^2(\Omega)}||\nabla \psi||_{L^2(\Omega)}\,ds\\ 
    &\,+ \int_{0}^{t} || u||_{L^4(\Omega)} ||\nabla w||_{L^4(\Omega)} ||\psi||_{L^4(\Omega)}\, ds+ \int_{0}^{t}  ||\theta||_{L^2(\Omega)} ||\psi||_{L^2(\Omega)}\,ds.
    \end{align*}
Hence from step 1, it follows that
\begin{align*}
    \langle \nabla P(t), \psi\rangle \leq C||\psi||_{H^1_0(\Omega)}.
\end{align*}
\end{proof}

\begin{thm}\label{thm 3.3}
    Let $(v_1,u_1,w_1)$ and $(v_2,u_2,w_2)$ be two solutions of $\mathcal{(P)}$ for given $\theta_1$, $\theta_2 \in \mathcal{U}$, respectively. Then the following estimate holds:
\begin{align}
    || u_1-u_2||_{L^2(0,T; H^3(\Omega))}+ || \partial_t u_1- \partial_t u_2||_{L^2(0,T; H^1(\Omega)^*)}+||\nabla w_1-\nabla w_2||_{L^2((0,T)\times  \Omega)} \nonumber\\
    +||v_1-v_2||_{L^2(0,T; \mathcal{V})}
    +|| \nabla u_1-\nabla u_2||_{L^2((0,T)\times  \Omega)}+ || \partial_t v_1- \partial_t v_2||_{L^2(0,T; \mathcal{V}^*)}\nonumber\\
    \leq C ||\theta_1- \theta_2||_{L^2((0,T)\times   \Omega)},\label{2.28}
\end{align}
where the constant $C>0$ depends on $K,\, T$ and the initial data of the system $\mathcal{(P)}$.
\end{thm}
\begin{proof}
Let $(v_1,u_1,w_1)$ and $(v_2,u_2,w_2)$ be two solutions of $\mathcal{(P)}$ with the same initial condition and given $\theta_1$, $\theta_2$, respectively. Define $v:= v_1-v_2$, $u:=u_1-u_2$, $w:=w_1-w_2$ and $\theta:=\theta_1-\theta_2 $  for $t \in (0,T)$, then $(v, u, w)$ satisfies
\begin{subequations}
\begin{align}
&\partial_t v- \mu\Delta v\;\;=\;\; - \lambda u_1 \nabla w- \lambda u \nabla w_2 + \theta && \textnormal { in }\quad (0,T)\times \Omega,\label{2.29}\\
&\hspace{0.9cm}\nabla \cdot v\;\;=\;\;0 && \textnormal { in }\quad (0,T)\times \Omega,\label{2.30}\\
&\hspace{1.5cm}v\;\;=\;\;0 && \textnormal { on }\quad (0,T)\times \partial \Omega,\label{2.31}\\
&\hspace{0.9cm}v (0,x) =0 && \text{in }\quad \Omega,\label{2.32}\\
&\,\partial_t u +  v_1 \cdot \nabla u+  v \cdot \nabla u_2+\alpha u  =  \Delta w && \text{in }\quad (0,T)\times \Omega,\label{2.33}\\
&\,w =-\Delta u+f(u_1)-f(u_2) && \text{in }\quad (0,T)\times \Omega, \label{2.34}\\
&\hspace{0.9cm}\nabla u \cdot \vec{n} =0 && \text{on } \quad (0,T)\times \partial\Omega,\label{2.35}\\
&\hspace{0.9cm}\nabla w \cdot \vec{n}=0 && \text{on }\quad (0,T)\times \partial\Omega,\label{2.36}\\
&\hspace{0.9cm}u(0,x)=0 && \text{in }\quad \Omega.\label{2.37}
\end{align}
\end{subequations}
Multiplying \eqref{2.29} and \eqref{2.33} by $v$ and $-\lambda \Delta u$, respectively and integrating the resulting equalities, we obtain
    \begin{align*}
        &\frac{1}{2}  \frac{d}{dt}||v||^2_{L^2(\Omega)}+ \mu ||\nabla v||^2_{L^2(\Omega)}+\frac{\lambda}{2}\frac{d}{dt}||\nabla u||^2_{L^2(\Omega)}\\
        =& -\lambda \int_{\Omega}(\Delta w - v_1 \cdot \nabla u-  v \cdot \nabla u_2-\alpha u)\Delta u\,dx+ \int_{\Omega} \theta \cdot v\,dx+\lambda \int_{\Omega} v \cdot \nabla u w_1\,dx\\
        & \hspace{5cm}+ \lambda \int_{\Omega} v \cdot \nabla u_2(-\Delta u+f(u_1)-f(u_2)) \,dx\\
        =&-\lambda \int_{\Omega} \Delta w \Delta u\,dx+ \lambda \int_{\Omega} v_1 \cdot \nabla u \Delta u\,dx+ \lambda \alpha \int_{\Omega} u \Delta u\, dx+ \int_{\Omega} \theta \cdot v\, dx\\
        &\,\hspace{2cm} + \lambda \int_{\Omega} v \cdot \nabla u w_1\,dx+ \lambda \int_{\Omega} v \cdot \nabla u_2(f(u_1)-f(u_2)) \,dx \\
        = &\, -\lambda \int_{\Omega} \Delta u \Delta(-\Delta u+f(u_1)-f(u_2))\,dx+  \lambda \int_{\Omega} v_1 \cdot \nabla u \Delta u\,dx\nonumber\\
        &+ \lambda \alpha \int_{\Omega} u \Delta u\, dx+ \int_{\Omega} \theta \cdot v\, dx+ \lambda \int_{\Omega} v \cdot \nabla u w_1\,dx+ \lambda \int_{\Omega} v \cdot \nabla u_2(f(u_1)-f(u_2)) \,dx\\
        =&-\lambda ||\nabla \Delta u||^2_{L^2(\Omega)}+ \lambda\int_{\Omega} \nabla (f(u_1)-f(u_2))\cdot \nabla \Delta u\,dx+  \lambda \int_{\Omega} v_1 \cdot \nabla u \Delta u\,dx\\
        &+ \lambda \alpha \int_{\Omega} u \Delta u\, dx+ \int_{\Omega} \theta \cdot v\, dx+ \lambda \int_{\Omega} v \cdot \nabla u w_1\,dx+ \lambda \int_{\Omega} v \cdot \nabla u_2(f(u_1)-f(u_2)) \,dx\\
         \leq & -\lambda ||\nabla \Delta u||^2_{L^2(\Omega)}+ \lambda|| \nabla (f(u_1)-f(u_2))||_{L^2(\Omega)}|| \nabla \Delta u||_{L^2(\Omega)} + \alpha \lambda || \nabla u||^2_{L^2(\Omega)}\\
         &+ \lambda|| \nabla \Delta u||_{L^2(\Omega)} ||v_1||_{L^4(\Omega)} || u||_{L^4(\Omega)}
       +\lambda ||\nabla w_1||_{L^2(\Omega)}||v||_{L^4(\Omega)}||u||_{L^4(\Omega)}\\
        &+||\theta||_{L^2(\Omega)}|| v||_{L^2(\Omega)}+ \lambda || v||_{L^4(\Omega)} || u_2||_{L^4(\Omega)}||\nabla (f(u_1)-f(u_2))||_{L^2(\Omega)}.
    \end{align*}
  Now,
  \begin{align*}
      &||\nabla (f(u_1)-f(u_2))||_{L^2(\Omega)}\\
      =& \,||f'(u_1)\nabla u_1- f'(u_2) \nabla u_2||_{L^2(\Omega)}\\
      \leq&\,|| f'(u_1) \nabla u||_{L^2(\Omega)}+ || (f'(u_1)-f'(u_2)) \nabla u_2||_{L^2(\Omega)} \\
      \leq&\,||f'(u_1)||_{L^3(\Omega)}|| \nabla u||_{L^6(\Omega)}+3||u_1+u_2||_{L^6(\Omega)}||u||_{L^6(\Omega)} ||\nabla u_2||_{L^6(\Omega)}\\
      \leq &\, C_1 ||f'(u_1)||_{L^3(\Omega)} || \nabla u||^{\frac{1}{2}}_{L^2(\Omega)}|| \nabla \Delta u||^{\frac{1}{2}}_{L^2(\Omega)}\\
      &\hspace{1cm}+ C_2||u_1+u_2||_{L^6(\Omega)}||\nabla u_2||_{L^6(\Omega)}|| \nabla u||_{L^2(\Omega)}.
  \end{align*}
  This implies that
  \begin{align*}
      &\frac{1}{2}  \frac{d}{dt}||v||^2_{L^2(\Omega)}+ (\mu-\lambda) ||\nabla v||^2_{L^2(\Omega)}+\frac{\lambda}{2}\frac{d}{dt}||\nabla u||^2_{L^2(\Omega)}+ \frac{\lambda}{2} ||\nabla \Delta u||^2_{L^2(\Omega)}\nonumber\\
      \leq&\, C_4 \Big(||v||^2_{L^2(\Omega)}+\lambda ||\nabla u||^2_{L^2(\Omega)}\Big)+ \frac{1}{2}|| \theta||^2_{L^2(\Omega)},\label{2.38}
  \end{align*}
  where 
  \begin{align*}
      C_4=C\Big\{\Big(1+||u_2||^2_{L^4(\Omega)}\Big)\Big(||f'(u_1)||^4_{L^2(\Omega)}+ ||u_1 +u_2||^2_{L^6(\Omega)}||\nabla u_2||^2_{L^2(\Omega)}\Big)\\
      + ||v_1||^2_{L^4(\Omega)}+ ||\nabla w_1||^2_{L^2(\Omega)}\Big\}.
  \end{align*}
  
  Choose $\mu > \lambda$, then Gronwall's inequality yields
  \begin{align*}
    ||v(t)||^2_{L^2(\Omega)}+  ||\nabla u(t)||^2_{L^2(\Omega)}+ \int_{0}^{T}\Big(2(\mu-\lambda) ||\nabla v||^2_{L^2(\Omega)}+ \lambda||\nabla \Delta u||^2_{L^2(\Omega)}\Big) \,dt \\ \leq \Big(\int_{0}^{T}|| \theta||^2_{L^2(\Omega)}\,dt]\Big) exp\Big(\int_{0}^{T} C_4(t)\,dt\Big),
  \end{align*}
  where 
  \begin{align*}
      \int_{0}^{T}C_4(t)dt =& C\int_{0}^{T}\Big\{\Big(1+||u_2||^2_{L^4(\Omega)}\Big)\Big(||f'(u_1)||^4_{L^2(\Omega)}+ ||u_1 +u_2||^2_{L^6(\Omega)}||\nabla u_2||^2_{L^2(\Omega)}\Big)\\
      & \hspace{5.5cm}+ ||v_1||^2_{L^4(\Omega)}+ ||\nabla w_1||^2_{L^2(\Omega)}\Big\}dt.\\
      \leq & C(||v_0||_{L^2(\Omega)^n},||u_0||_{H^1(\Omega)},K, T).
  \end{align*}
   Hence,
   \begin{align*}
   &\nabla w= -\nabla \Delta u+\nabla ( f(u_1)-f(u_2))\\
      \Rightarrow& \int_{0}^{T} || \nabla w||^2_{L^2(\Omega)}\,dt\leq 2\int_{0}^{T}\Big(|| \nabla \Delta u||^2_{L^2(\Omega)}+ ||  \nabla ( f(u_1)-f(u_2))||^2_{L^2(\Omega)}\Big)\,dt \\
      &\hspace{3cm} \leq C_5\int_{0}^{T}|| \theta||^2_{L^2(\Omega)}\,dt.
   \end{align*}
     Furthermore, 
    \begin{align*}
        \int_{0}^{T} || \partial_t u||^2_{H^1(\Omega)^*}\, dt \leq &\,2\int_{0}^{T}\Big( || v_1||^2_{L^6(\Omega)}|| \nabla u||^2_{L^2(\Omega)}+ ||v||^2_{L^6(\Omega)}||\nabla u_2||^2_{L^2(\Omega)}\\
       &\hspace{4cm}+\alpha||u||^2_{L^2(\Omega)} +|| \nabla w||^2_{L^2(\Omega)}\Big)\,dt \\
        &\leq\, C_6 \int_{0}^{T}|| \theta||^2_{L^2(\Omega)}\,dt
    \end{align*}
    and
    \begin{align*}
        \int_{0}^{T} || \partial_t v||^2_{\mathcal{V}^*}\, dt \leq &\,2\int_{0}^{T}\Big(  \mu||\nabla v||^2_{L^2(\Omega)}+  \lambda ||u_1||^2_{L^6(\Omega)} ||\nabla w||^2_{L^2(\Omega)} \\
    &\,\hspace{1cm}+||u||^2_{L^6(\Omega)}|| \nabla w_2||^2_{L^2(\Omega)}+ ||\theta||^2_{L^2(\Omega)}\Big)\, dt \\
     \leq &\, C_7 \int_{0}^{T}|| \theta||^2_{L^2(\Omega)}\,dt,
    \end{align*}
    where the constants $C_5,C_6,C_7$ are depending on $C_4$. 
    \end{proof}
Putting these results together, one can simply conclude that control to state operator i.e., $\mathcal{S}: \theta \to (v,u,w)$ is a Lipschitz continuous from $L^2(0,T;\mathcal{H})$ to
\begin{align*}
    \mathcal{F}= L^2(0,T; \mathcal{V}) \cap H^1(0,T; \mathcal{V}^*) \times H^1(0,T;H^1(\Omega)^*) \cap C(0,T; H^1(\Omega)) \cap L^2(0,T; H^3(\Omega)) \\\times L^2(0,T;H^1(\Omega)).
\end{align*}
Hence the control to state mapping $\mathcal{S}: \theta \to (v,u,w)$ is well-defined from $L^2(0,T; \mathcal{H})$ to $\mathcal{F}$. Moreover, $\mathcal{S}$ is a Lipschitz continuous mapping from $\mathcal{U}$ of $L^2(0,T;\mathcal{H})$ to $\mathcal{F}.$


 \section{Optimality condition}
\begin{lemma}
   Assume that assumptions (A1) -- (A4) hold, then $\mathcal{(OP)}$ admits a solution.
\end{lemma}
\begin{proof}
    Let $m= \underset{\theta \in \mathcal{U}_{ad}}{inf} \,J(\theta)$. Since $0 \leq m < \infty$, then, there exists a sequence $\{\theta_n\}_n \in \mathcal{U}_{ad} \subset L^2(0,T;\mathcal{H})$ s.t. $\displaystyle\lim_{n \to \infty}\, J(\theta_n) =m$. Further assume that $\mathcal{S}(\theta_n)=(v_n,u_n,w_n)$ for $n \in \mathbb{N}$. Hence, from Theorem \ref{Theorem 3.1}, there exists subsequences of $\{\theta_n\}_n,\, \{v_n\}_n,\, \{\partial_t v_n\}_n,\, \{u_n\}_n,\, \{\partial_t u_n\}_n,\, \{w_n\}_n$ still indexed by $n$ s.t.

    \begin{itemize}
        \item[(i)] $\theta_n \overset{w}{\rightharpoonup} \theta$ in $\mathcal{U}_{ad} \subset$  \,$L^2(0,T;\mathcal{H})$,
        \item[(ii)] $v_n \overset{w}{\rightharpoonup} v$ in $L^2(0,T; \mathcal{V})$,
        \item[(iii)] $\partial_t v_n \overset{w}{\rightharpoonup} \partial_t v$ in $L^2(0,T; \mathcal{V}^*)$,
        \item[(iv)] $u_n \overset{w}{\rightharpoonup} u$ in $L^2(0,T; H^3(\Omega))$, 
        \item[(v)] $u_n \overset{w^*}{\rightharpoonup} u$ in $L^\infty(0,T; H^1(\Omega))$,
        \item[(vi)] $\partial_t u_n \overset{w}{\rightharpoonup} \partial_t u$ in $L^2(0,T;H^1(\Omega)^*)$,
        \item[(vii)] $w_n \overset{w}{\rightharpoonup} \hat{w}$ in $L^2(0,T;H^1(\Omega))$.
    \end{itemize}

Again, applying Lion-Aubin's compactness theorem yields
\begin{equation*}
    u_n \to u \text{  strongly in } L^2(0,T;H^2(\Omega)) \cap C([0,T];H^1(\Omega)).
\end{equation*}
This implies that
\begin{align*}
    \hat{w}= -\Delta u+f(u)= w
\end{align*}
and 

\begin{align*}
   & \int_{(0,T)\times  \Omega} (v_n u_n) \cdot \nabla \phi_1 \,\di(x,t)\to  \int_{(0,T)\times  \Omega} (vu)\cdot \nabla\phi_1\, \di(x,t)\,\, \text{as} \,\,  n \to \infty,\\
    & \int_{(0,T)\times  \Omega}\phi_2 \cdot (\nabla u_n w_n) \,\di(x,t) \to \int_{(0,T)\times  \Omega}\phi_2 \cdot (\nabla u w) \,\di(x,t) \,\, \text{as} \,\,  n \to \infty
\end{align*}
for $\phi_1 \in \mathcal{V}$ and $\phi_2 \in H^1(\Omega)$. Hence, $\mathcal{S}(\theta)=(v,u,w) \in \mathcal{F}$. Next, the weakly lower semicontinuity of $J$ implies that $\theta$ is an optimal control for $\mathcal{(OP)}$.
\end{proof}
  Throughout this article, we frequently use $\hat{\theta}$ as a local optimal control along with the associate state $(\hat{v},\hat{u}, \hat{w})= \mathcal{S}(\hat{\theta}) \in \mathcal{F}$. 
    
  \subsection{Study of the linearized system}
  Let $h$ be a fixed control. Now to get the Fr\,echet differentiability of the control to state operator $\mathcal{S}$, we want to find the linearized form of the state equations $\mathcal{(P)}$ around $(\hat{ v},\hat{u}, \hat{w})$. Taking $u=\hat{u}+\varphi_2,\, w=\hat{w}+\varphi_3,\,  v=\hat{ v}+\varphi_1, \, p=\hat{p}+ \bar{p}$ and substituting this we obtain the following linearized system  

    \begin{subequations}
        \begin{align}
            &\frac{\partial \varphi_1}{\partial t}- \mu \Delta \varphi_1+ \nabla \bar{p
            }=- \lambda \varphi_2 \nabla \hat{w}- \lambda\hat{u} \nabla \varphi_3 + h&&\text{in}\quad (0,T)\times  \Omega,\label{lin1}\\ 
            &\frac{\partial \varphi_2}{\partial t}- \Delta \varphi_3 + \varphi_1 \cdot \nabla \hat{u}+\hat{ v} \cdot \nabla \varphi_2+ \alpha \varphi_2=0 &&\text{in}\quad (0,T)\times  \Omega,\label{lin2}\\
            &\varphi_3= -\Delta \varphi_2+ f'(\hat{u})\varphi_2 && \text{in}\quad (0,T)\times  \Omega,\label{lin3}\\
            &\nabla \cdot \varphi_1=0&& \text{in}\quad (0,T)\times  \Omega,\label{lin4}\\
            &\varphi_1(x,t)=0,\, \nabla \varphi_2 \cdot \vec{n}=0,\,\nabla \varphi_3 \cdot \vec{n}=0&& \text{on}\quad (0,T)\times  \partial\Omega,\label{lin5}\\
            &\varphi_1(x,0)=0,\, \varphi_2(x,0)=0&& \text{in}\quad\Omega.\label{lin6}
        \end{align}
    \end{subequations}
\begin{thm}{\label{thm 4.2}}
    For $h \in L^2(0,T;\mathcal{H})$, the linearized problem \eqref{lin1} -- \eqref{lin6} admits a unique weak solution $(\varphi_1,\varphi_2,\varphi_3)$ such that the following estimate holds 
    \begin{align*}
     || \varphi_2(t)||^2_{H^1(\Omega)}+ \int_{0}^{t} \Big(||\varphi_2(t)||^2_{H^3(\Omega)}+ || \varphi_3(t)||^2_{H^1(\Omega)}+ ||\varphi_1(t)||^2_{H^1(\Omega)}\Big)\,dt\\
     \leq C \int_{0}^{t}||h(t)||^2_{L^2(\Omega)}\,dt
    \end{align*}
    for a.e. $t \in (0,T)$.
\end{thm}

    \begin{proof}
We begin with the existence of weak solution for the linearized problem \eqref{lin1} -- \eqref{lin6}. We are going to use the same procedure as Theorem \ref{thm 3.2}.
Let $\{\eta_i\}_{i=1}^\infty$ be a family of eigenfunctions  to the Stokes operator $A$ with the corresponding non-decreasing sequence of eigenvalues $\{\kappa_i\}_{i=1}^\infty$ such that $A\eta_i=\kappa_i \eta_i$.  $\{\kappa_i\}_{i=1}^\infty$ are complete orthonormal basis of $\mathcal{H}$ and orthogonal to $\mathcal{V}$. Further assume that  $\{\xi_i\}$ be the orthonormal base in $L^2(\Omega)$ being composed of the eigenfunctions of the operator $-\Delta$ with the Neumann boundary condition. 
\par Next, take $X_N=$ span $\{\xi_1, \xi_2,\cdots, \xi_N\}$, $Y_N=$ span $\{\eta_1, \eta_2, \cdots, \eta_N\}$. Let $P_{X_N}$ and $P_{Y_N}$ be the orthogonal projectors from $L^2(\Omega)$ to $X_N$ and $\mathcal{H}$ to $Y_N$, respectively. Then we look for functions 
\begin{align*}
    (\varphi_2)_N=\displaystyle \sum_{i=1}^N a_{N_i}(t) \xi_i \in X_N,\, (\varphi_3)_N=\displaystyle \sum_{i=1}^N b_{N_i}(t) \xi_i \in X_N\, \text{and} \,(\varphi_1)_N=\displaystyle \sum _{i=1}^N c_{N_i}(t)\eta_i \in Y_N
\end{align*}
solve the following approximate problem
\begin{align*}
    &  \langle\partial_t  (\varphi_1)_N, \eta \rangle +  \mu \int_{\Omega} \nabla  (\varphi_1)_N  : \nabla \eta \, dx =-\lambda \int_{ \Omega} (\varphi_2)_N \nabla \hat{w} \cdot \eta\,dx- \lambda \int_{ \Omega}\hat{u}\nabla (\varphi_3)_N\cdot \eta \,dx\\
    &\hspace{10cm} +  \int_{\Omega} h \cdot \eta_i\,dx,\\
     &  \langle\partial_t (\varphi_2)_N, \xi \rangle \,+ \int_{\Omega} \nabla (\varphi_3)_N \cdot \nabla \xi\, dx + \alpha \int_{ \Omega} (\varphi_2)_N \xi \,dx + \int_{\Omega} (\varphi_1)_N \cdot (\nabla \hat{u}) \xi\,dx\\
    &\hspace{8cm}+ \int_{\Omega} \hat{v}\cdot \nabla (\varphi_2)_N\xi\,dx   = 0,\\
     &\int_{\Omega}(\varphi_3)_N \zeta \,dx =\int_{ \Omega}\nabla (\varphi_2)_N \cdot \nabla \zeta \, dx+ \int_{\Omega} f'(\hat{u})(\varphi_2)_N\zeta\, dx
\end{align*}
for all $\eta \in Y_N$ and $\xi, \zeta \in X_N$ with $a_{N_i}(0)=\langle \varphi_2(x,0), \xi_i \rangle$ and $c_{N_i}(0)=\langle \varphi_1(x,0), \eta_i \rangle$; $i=1,2 \cdots N$. Note that $(\varphi_2)_N(0,x)=0$ and $(\varphi_1)_N(0,x)=0$.

Choose $\eta =(\varphi_1)_N,\, \xi=\lambda (\varphi_3)_N$ and $\zeta=\lambda \frac{\partial (\varphi_2)_N}{\partial t}$, then we get
\begin{align}
    &\frac{1}{2}\frac{d }{dt}||(\varphi_1)_N||^2_{L^2(\Omega)} + \mu || \nabla (\varphi_1)_N||^2_{L^2(\Omega)}+ \frac{\lambda}{2} \frac{d }{dt} || \nabla (\varphi_2)_N||^2_{L^2(\Omega)}+ \lambda || \nabla (\varphi_3)_N||^2_{L^2(\Omega)}\nonumber\\
   = &\, -\lambda \int_{\Omega} (\varphi_2)_N \nabla \hat{w}\cdot (\varphi_1)_N\,dx+ \int_{\Omega} h \cdot (\varphi_1)_N \,dx +  \lambda  \int_{\Omega} \hat{v} \cdot \nabla (\varphi_3)_N  (\varphi_2)_N \,dx\nonumber\\
    &\hspace{4cm}- \lambda \int_{\Omega} f'(\hat{u})(\varphi_2)_N \frac{\partial (\varphi_2)_N}{\partial t}\,dx -\lambda \alpha \int_{\Omega} (\varphi_2)_N (\varphi_3)_N\, dx\nonumber\\
   \leq &\,\,  \lambda || (\varphi_2)_N||_{L^6(\Omega)} || \nabla \hat{w}||_{L^2(\Omega)}|| (\varphi_1)_N||_{L^3(\Omega)} + \lambda || \hat{v}||_{L^6(\Omega)}|| \nabla (\varphi_3)_N||_{L^2(\Omega)}|| (\varphi_2)_N||_{L^3(\Omega)}\nonumber\\
   &+ ||h||_{L^2(\Omega)} || (\varphi_1)_N||_{L^2(\Omega}- \lambda \int_{\Omega} f'(\hat{u})(\varphi_2)_N \frac{\partial (\varphi_2)_N}{\partial t}\,dx -\lambda \alpha\int_{\Omega} (\varphi_2)_N (\varphi_3)_N\, dx.\label{lin7}
\end{align}
Now,
\begin{align*}
    \int_{\Omega} (\varphi_2)_N (\varphi_3)_N\, dx= &\int_{\Omega} (\varphi_2)_N (-\Delta (\varphi_2)_N+ f'(\hat{u})(\varphi_2)_N)\,dx\\
    \leq & \,\,||\nabla (\varphi_2)_N||^2_{L^2(\Omega)}+ \frac{1}{3} || (\varphi_2)_N||^2_{L^4(\Omega)}||\hat{u}||^2_{L^4(\Omega)}+ || (\varphi_2)_N||^2_{L^2(\Omega)}.
\end{align*}
Thanks to
\begin{align*}
    \int_{\Omega} \Big(\frac{\partial (\varphi_2)_N}{\partial t}+\alpha (\varphi_2)_N\Big) \,dx=0
\end{align*}
with $(\varphi_2)_N(0,x)=0$ implies that
\begin{equation*}
    \int_{\Omega} (\varphi_2)_N(x,t)\,dx = 0.
\end{equation*}
Hence, 
\begin{align*}
    \int_{\Omega} (\varphi_2)_N (\varphi_3)_N\, dx \leq C || \nabla (\varphi_2)_N||^2_{L^2(\Omega)}.
\end{align*}
Again,
\begin{align*}
    &-\int_{\Omega} f'(\hat{u})(\varphi_2)_N \frac{\partial (\varphi_2)_N}{\partial t}\,dx= -\frac{1}{2} \frac{d}{dt}\int_{\Omega}f'(\hat{u})| (\varphi_2)_N|^2 \,dx+ 3 \int_{\Omega} \hat{u} \partial_t \hat{u} |(\varphi_2)_N|^2\,dx
    \end{align*}
    Hence,
    \begin{align*}
    &\, -\int_{0}^{t} \int_{\Omega} f'(\hat{u})(\varphi_2)_N \frac{\partial (\varphi_2)_N}{\partial t}\,\di(x,t)\\
    = &-\frac{1}{2} \int_{\Omega}f'(\hat{u}(t))| (\varphi_2(t))_N|^2 \,dx + 3\int_{0}^{t}  \int_{\Omega} \hat{u} \partial_t \hat{u} |(\varphi_2)_N|^2\,\di(x,t)\\
   \leq & \,\,\frac{1}{2} ||f'(\hat{u}(t))||_{L^2(\Omega)} || (\varphi_2)_N||^2_{L^4(\Omega)}+ 3\int_{0}^{t}  \int_{\Omega} \hat{u} \partial_t \hat{u} |(\varphi_2)_N|^2\,\di(x,t)
    \end{align*}
    and 
    \begin{align*}
        &\int_{0}^{t}  \int_{\Omega} \hat{u} \partial_t \hat{u} |(\varphi_2)_N|^2\,\di(x,t)\\
        \leq&\, \int_{0}^{t} ||\partial_t \hat{u}||_{H^1(\Omega)^*} || \hat{u} |(\varphi_2)_N|^2||_{H^1(\Omega)}\,dt\\
        \leq&\,\int_{0}^{t} ||\partial_t \hat{u}||_{H^1(\Omega)^*} \Big ( ||\hat{u} |(\varphi_2)_N|^2||_{L^2(\Omega)}+|| \nabla (\hat{u} |(\varphi_2)_N|^2)||_{L^2(\Omega)}\Big)\,dt\\
        \leq &\, \int_{0}^{t} ||\partial_t \hat{u}||_{H^1(\Omega)^*} \Big ( ||\hat{u}||_{L^6(\Omega)}|| (\varphi_2)_N||^2_{L^6(\Omega)}+ || \nabla \hat{u}||_{L^6(\Omega)}||(\varphi_2)_N||^2_{L^6(\Omega)}\\
        & \hspace{3cm} +||\hat{u}||_{L^6(\Omega)}|| (\varphi_2)_N||_{L^6(\Omega)}|| \nabla (\varphi_2)_N||_{L^6(\Omega)}\Big)\,dt.
    \end{align*}

Integrating \eqref{lin7} w.r.t. $t$ yields
\begin{align}
    & \frac{1}{2}||(\varphi_1)_N(t)||^2_{L^2(\Omega)} +\frac{\lambda}{2}  || \nabla (\varphi_2)_N(t)||^2_{L^2(\Omega)}+\int_{0}^{t} \Big( \mu || \nabla (\varphi_1)_N||^2_{L^2(\Omega)}+ \lambda || \nabla (\varphi_3)_N||^2_{L^2(\Omega)}\Big)\,dt\nonumber\\
     \leq &\,\,  \lambda\int_{0}^{t} \Big(|| (\varphi_2)_N||_{L^6(\Omega)} || \nabla \hat{w}||_{L^2(\Omega)}|| (\varphi_1)_N||_{L^3(\Omega)} +  || \hat{v}||_{L^6(\Omega)}|| \nabla (\varphi_3)_N||_{L^2(\Omega)}|| (\varphi_2)_N||_{L^3(\Omega)}\Big)\,dt\nonumber\\
   & + \int_{0}^{t}||h||_{L^2(\Omega)} || (\varphi_1)_N||_{L^2(\Omega}\,dt- \lambda \int_{0}^{t}\int_{\Omega} f'(\hat{u})(\varphi_2)_N \frac{\partial (\varphi_2)_N}{\partial t}\,\di(x,t) -\lambda \alpha\int_{0}^{t} \int_{\Omega} (\varphi_2)_N (\varphi_3)_N\, \di(x,t)\nonumber\\
    \leq &\,\, \lambda\int_{0}^{t} \Big(|| (\varphi_2)_N||_{L^6(\Omega)} || \nabla \hat{w}||_{L^2(\Omega)}|| (\varphi_1)_N||_{L^3(\Omega)} +  || \hat{v}||_{L^6(\Omega)}|| \nabla (\varphi_3)_N||_{L^2(\Omega)}|| (\varphi_2)_N||_{L^3(\Omega)}\Big)\,dt\nonumber\\
   & + \int_{0}^{t} ||h||_{L^2(\Omega)} || (\varphi_1)_N||_{L^2(\Omega}\,dt+ \lambda C\int_{0}^{t} || \nabla (\varphi_2)_N||^2_{L^2(\Omega)}\,dt -\frac{\lambda}{2} \int_{\Omega}f'(\hat{u}(t))| (\varphi_2(t))_N|^2 \,dx \nonumber\\
   &+\lambda \int_{0}^{t} ||\partial_t \hat{u}||_{H^1(\Omega)^*} \Big ( ||\hat{u}||_{L^6(\Omega)}|| (\varphi_2)_N||^2_{L^6(\Omega)}+ || \nabla \hat{u}||_{L^6(\Omega)}||(\varphi_2)_N||^2_{L^6(\Omega)}\Big)\,dt\nonumber\\
   &\hspace{2.6cm}+ \lambda\int_{0}^{t} ||\partial_t \hat{u}||_{H^1(\Omega)^*} ||\hat{u}||_{L^6(\Omega)}|| (\varphi_2)_N||_{L^6(\Omega)}|| \nabla (\varphi_2)_N||_{L^6(\Omega)}\,dt. \label{lin8}
\end{align}
Again, choose $\xi= 4\lambda (\varphi_2)_N$ and $\zeta=-\lambda \Delta (\varphi_2)_N $ to get 
\begin{align}
    &\frac{4\lambda}{2}\frac{d}{dt}|| (\varphi_2)_N||^2_{L^2(\Omega}+\lambda  || \Delta (\varphi_2)_N||^2_{L^2(\Omega)}\nonumber\\
    \leq&  \,\,\lambda||f'(\hat{u})||_{L^3(\Omega)} || (\varphi_2)_N||_{L^6(\Omega)}|| \Delta (\varphi_2)_N||_{L^2(\Omega)}\nonumber\\
    &\hspace{2cm} + 4 \lambda || (\varphi_2)_N||^2_{L^2(\Omega)}+  4\lambda ||(\varphi_1)_N||_{L^2(\Omega)}||\nabla \hat{u}||_{L^6(\Omega)}||(\varphi_2)_N||_{L^3(\Omega)} \nonumber\\
   \Rightarrow & \, 2\lambda|| (\varphi_2)_N(t)||^2_{L^2(\Omega)}+\lambda\int_{0}^{t} || \Delta (\varphi_2)_N||^2_{L^2(\Omega)}\,dt \nonumber\\
   \leq& \,\, \lambda \int_{0}^{t}\Big(   ||f'(\hat{u})||_{L^3(\Omega)} || (\varphi_2)_N||_{L^6(\Omega)}|| \Delta (\varphi_2)_N||_{L^2(\Omega)}\nonumber\\
    &\hspace{2cm}  +4|| (\varphi_2)_N||^2_{L^2(\Omega)}+ 4 ||(\varphi_1)_N||_{L^2(\Omega)}||\nabla \hat{u}||_{L^6(\Omega)}||(\varphi_2)_N||_{L^3(\Omega)}\Big) \,dt.\label{lin9}
\end{align}
Adding \eqref{lin8} and \eqref{lin9}
\begin{align*}
      \Big(||(\varphi_1)_N(t)||^2_{L^2(\Omega)} +  \lambda || \nabla (\varphi_2)_N(t)||^2_{L^2(\Omega)}+\lambda || (\varphi_2)_N(t)||^2_{L^2(\Omega)}\Big)\\+2(\mu-\lambda)\int_{0}^{t}  || \nabla (\varphi_1)_N||^2_{L^2(\Omega)}\,dt
     + \int_{0}^{t}\Big(\lambda || \nabla (\varphi_3)_N||^2_{L^2(\Omega)}+||\Delta (\varphi_2)_N||^2_{L^2(\Omega)}\Big)\,dt\\
     \leq  \int_{0}^{t}\mathcal{C}(t)\Big(||(\varphi_1)_N(t)||^2_{L^2(\Omega)} + \lambda || \nabla (\varphi_2)_N(t)||^2_{L^2(\Omega)}+\lambda|| (\varphi_2)_N(t)||^2_{L^2(\Omega)}\Big)dt \\
     + C \int_{0}^{t}  || h(t)||^2_{L^2(\Omega)}\,dt,
\end{align*}
where 
\begin{align*}
    \mathcal{C}= C(1+ || \nabla \hat{w}||^2_{L^2(\Omega)}+ ||\hat{v}||^2_{L^6(\Omega)}+||f'(\hat{u})||^2_{L^3(\Omega)}+ || \partial_t \hat{u}||^2_{H^1(\Omega)^*}||\hat{u}||^2_{L^6(\Omega)}\\+||\nabla \hat{u}||_{L^6(\Omega)}
    +  || \partial_t \hat{u}||_{H^1(\Omega)^*}||\hat{u}||_{L^6(\Omega)}+||\nabla \hat{u}||_{L^6(\Omega)}).
\end{align*}
Choose $\mu > \lambda$ and applying Gronwall's inequality, we obtain
\begin{align*}
    ||(\varphi_1)_N(t)||^2_{L^2(\Omega)} +  \lambda|| \nabla (\varphi_2)_N(t)||^2_{L^2(\Omega)}+2(\mu-\lambda)\int_{0}^{t}  || \nabla (\varphi_1)_N||^2_{L^2(\Omega)}\,dt\\
     +|| \lambda (\varphi_2)_N(t)||^2_{L^2(\Omega)}+ \int_{0}^{t}\Big(\lambda || \nabla (\varphi_3)_N||^2_{L^2(\Omega)}+||\Delta (\varphi_2)_N||^2_{L^2(\Omega)}\Big)\,dt\\
     \leq \,\, C \Big(\int_{0}^{t}  || h(t)||^2_{L^2(\Omega)}\,dt\Big) exp\Big(\int_{0}^{t} \mathcal{C}(t) dt\Big)
\end{align*}
for a.e. $t \in (0,T)$.
Again, thanks to 
\begin{align*}
    \int_{\Omega} (\varphi_3)_N\,dx= \int_{\Omega} f'(\hat{u}) (\varphi_2)_N\,dx \leq ||f'(\hat{u})||_{L^2(\Omega)} ||(\varphi_2)_N||_{L^2(\Omega)}< \infty,
\end{align*}
which implies $(\varphi_3)_N \in L^2(0,T; H^1(\Omega))$ and 
\begin{align*}
    || \nabla \Delta (\varphi_2)_N||_{L^2(\Omega)} \leq &\,\,|| \nabla (\varphi_3)_N||_{L^2(\Omega)}+|| \nabla (f'(\hat{u}) (\varphi_2)_N)||_{L^2(\Omega)}\\
    \leq & \,\,|| \nabla (\varphi_3)_N||_{L^2(\Omega)} + ||f''(\hat{u})||_{L^6(\Omega)} || \nabla \hat{u} ||_{L^6(\Omega)}|| (\varphi_2)_N||_{L^6(\Omega)}\\
    & \hspace{4cm}+ ||f'(\hat{u})||_{L^3(\Omega)} ||\nabla (\varphi_2)_N||_{L^6(\Omega)}.
\end{align*}

Furthermore,
\begin{align*}
     \int_{0}^{t}||\partial_t (\varphi_1)_N||_{\mathcal{V}^*}\,dt \leq \int_{0}^{t} \Big( \mu ||\nabla (\varphi_1)_N||_{L^2(\Omega)}+  \lambda || (\varphi_2)_N||_{L^6(\Omega)}|| \nabla \hat{w}||_{L^2(\Omega)}\\
      +  \lambda ||\hat{u}||_{L^6(\Omega)} || \nabla (\varphi_3)_N||_{L^2(\Omega_2)}+ || h||_{L^2(\Omega)}\Big)\,dt
\end{align*}
and 
\begin{align*}
    \int_{0}^{t}||\partial_t (\varphi_2)_N||_{H^1(\Omega)^*}\,dt \leq \int_{0}^{t} \Big(||\Delta (\varphi_3)_N||_{L^2(\Omega)}+|| (\varphi_1)_N||_{L^3(\Omega)}||\nabla \hat{u}||_{L^2(\Omega)}\\
    + || \hat{v}||_{L^3(\Omega)}|| \nabla (\varphi_2)_N||_{L^2(\Omega)}+ \alpha || (\varphi_2)_N||_{L^2(\Omega)}\Big)\,dt
\end{align*}
for a.e. $t \in (0,T)$.\\

Since all the boundedness are uniform, we can extract subsequences from $\{(\varphi_1)_N\}_N$, $\{\partial_t (\varphi_1)_N\}_N$, $\{(\varphi_2)_N\}_N$, $ \{\partial_t(\varphi_2)_N\}_N$, $\{(\varphi_3)_N\}_N$ still indexed by same symbol s.t.

  \begin{itemize}
      \item[(i)] $(\varphi_1)_N \overset{w}{\rightharpoonup} \varphi_1$ in $L^2(0,T;\mathcal{V})$,
      \item[(ii)]  $\partial_t( \varphi_1)_N \overset{w}{\rightharpoonup} \partial_t \varphi_1$ in $L^2(0,T;\mathcal{V}^*)$,
      \item[(iii)] $(\varphi_2)_N \overset{w}{\rightharpoonup} \varphi_2 $ in $L^2(0,T;H^3(\Omega))$,
      \item[(iv)]  $(\varphi_2)_N \overset{w^*}{\rightharpoonup} \varphi_2 $ in $L^\infty(0,T;H^1(\Omega))$,
      \item[(v)] $\partial_t(\varphi_2)_N  \overset{w}{\rightharpoonup}  \partial_t(\varphi_2)$ in $L^2(0,T;H^1(\Omega)^*)$,
      \item[(vi)]  $(\varphi_3)_N \overset{w}{\rightharpoonup} \varphi_3$ in $L^2(0,T;H^1(\Omega))$.
  \end{itemize} 

Now, passing these limits in approximate problem, we get a weak solution.\\

\par Finally, let $((\varphi_1)_1, (\varphi_2)_1, (\varphi_3)_1)$ and $((\varphi_1)_2, (\varphi_2)_2, (\varphi_3)_2)$ be two weak solutions of \eqref{lin1} -- \eqref{lin6}. Further assume that $\varphi_1=(\varphi_1)_1-(\varphi_1)_2$, $\varphi_2=(\varphi_2)_1-(\varphi_2)_2$ and $\varphi_3=(\varphi_3)_1-(\varphi_3)_2$, then $(\varphi_1, \varphi_2, \varphi_3)$ satisfies the following equations
\begin{align*}
 &\frac{\partial \varphi_1}{\partial t}- \mu \Delta \varphi_1+ \nabla \bar{p
            }=- \lambda \varphi_2 \nabla \hat{w}- \lambda\hat{u} \nabla \varphi_3 &&\text{in}\quad (0,T)\times  \Omega,\\ 
            &\frac{\partial \varphi_2}{\partial t}- \Delta \varphi_3 + \varphi_1 \cdot \nabla \hat{u}+\hat{ v} \cdot \nabla \varphi_2+ \alpha \varphi_2=0 &&\text{in}\quad (0,T)\times  \Omega,\\
            &\varphi_3= -\Delta \varphi_2+ f'(\hat{u})\varphi_2 && \text{in}\quad (0,T)\times  \Omega,\\
            &\nabla \cdot \varphi_1=0&& \text{in}\quad (0,T)\times  \Omega,\\
            &\varphi_1(x,t)=0,\, \nabla \varphi_2 \cdot \vec{n}=0,\,\nabla \varphi_3 \cdot \vec{n}=0&& \text{on}\quad (0,T)\times  \partial\Omega,\\
            &\varphi_1(x,0)=0,\, \varphi_2(x,0)=0&& \text{in}\quad\Omega,    
\end{align*}
where $\Bar{p}= \Bar{p}_{(\varphi_1)_1}-\Bar{p}_{(\varphi_1)_2}$. \\
\par Deriving a similar type of estimate of Theorem \ref{thm 4.2} for $h\equiv 0$, we get
\begin{equation*}
    (\varphi_1)_1=(\varphi_1)_2,\, (\varphi_2)_1=(\varphi_2)_2,\, (\varphi_3)_1=(\varphi_3)_2
\end{equation*}
for a.e. in $(0,T)\times  \Omega$.

\end{proof}

\subsection{Differentiability of control to state operator}
\begin{thm}
Let the assumptions (A1) -- (A4) ar satisfied. Then, for any $\hat{\theta} \in \mathcal{U}$, the control to sate mapping $\mathcal{S}$ is Fr\'echet differentiable in $\mathcal{U}$ as a mapping from $L^2(0,T; \mathcal{H})$ into the space $\mathcal{F}$. Moreover, for any $h \in L^2(0,T;\mathcal{H})$, its Fr\'echet derivative $\mathcal{DS}$ is given by
 \begin{equation*}
        \mathcal{DS}(\hat{\theta})(h)=(\varphi_1, \varphi_2, \varphi_3),
    \end{equation*}
where $(\varphi_1, \varphi_2, \varphi_3)$ is the weak solution to the linearized system \eqref{lin1} -- \eqref{lin6} with respect to $h$.

\end{thm}

\begin{proof}
For any fixed $\hat{\theta} \in \mathcal{U}$, let $\mathcal{S}(\hat{\theta})=(\hat{v},\hat{u},\hat{w}) $  be the associated solution to the state equations $\mathcal{(P)}$. Now, $\mathcal{U}$ is an open subset of of $L^2(0,T;\mathcal{H})$, then there exists some $\lambda>0$ such that for any $h \in L^2(0,T; \mathcal{H})$ with $||h||_{L^2(0,T;\mathcal{H})} \leq \lambda$, we have $\hat{\theta}+h \in \mathcal{U}$. For such $ h \in L^2(0,T;\mathcal{H})$, let 
$(v^h, u^h, w^h)$ be the solution of the system $\mathcal{(P)}$ with respect to $\hat{\theta}+h$. Define $v:=v^h-\hat{v}$, $u:=u^h- \hat{u}$ and $w:=w^h-\hat{w}$, then $(v,u,w)$ satisfies
\begin{align*}
    &\frac{\partial v}{\partial t}- \mu \Delta v+ \nabla p_v= -\lambda (u \nabla w + \hat{u} \nabla w+ u \nabla \hat{w})+h&& \text{in}\quad (0,T)\times  \Omega,\\ 
    &\nabla \cdot v\;\;=\;\;0	&& \textnormal { in }\quad (0,T)\times \Omega, \\
    &v\;\;=\;\;0 && \textnormal { on }\quad (0,T)\times \partial \Omega,
			\\
			 &v (0,x) 
			=  0
			&& \text{in }\quad \Omega,\\
    &\frac{\partial u}{\partial t}+  v \cdot \nabla u+  v \cdot \nabla \hat{u}+ \hat{v}\cdot \nabla u+\alpha u=\Delta w&& \text{in}\quad (0,T)\times  \Omega,\\
    &w= -\Delta u + f(u^h)-f(\hat{u})&& \text{in}\quad (0,T)\times  \Omega,\\
   & \nabla u \cdot \vec{n} =0
			&& \text{on } \quad (0,T)\times \partial\Omega,
			\\
			&\nabla w \cdot \vec{n}=0
			&& \text{on }\quad (0,T)\times \partial\Omega,
			\\
			&u(0,x)=0
			&& \text{in }\quad \Omega,
    \end{align*}
where $p_v= p_{v^h}-p_{\hat{v}}$.\par
Again, define $ y_1:= v-\varphi_1$, $y_2:=u-\varphi_2$ and $y_3:= w-\varphi_3$, then $(y_1,y_2,y_3)$ satisfies 
\begin{align}
    &\frac{\partial y_1}{\partial t}- \mu \Delta y_1+ \nabla p_{y_1}=-\lambda (u \nabla w+ \hat{u} \nabla y_3+ y_2 \nabla \hat{w})&& \text{in}\quad (0,T)\times  \Omega, \label{fre1}\\
    &\nabla \cdot y_1=0 && \textnormal { in }\quad (0,T)\times \Omega, \label{fre2}\\
  & y_1\;\;=\;\;0 && \textnormal { on }\quad (0,T)\times \partial \Omega, \label{fre3}\\
 &y_1 (0,x)=  0 && \text{in }\quad \Omega,\label{fre4}\\
 &\frac{\partial y_2}{\partial t} +  v \cdot \nabla u+  y_1 \cdot \nabla \hat{u} +  \hat{v} \cdot \nabla y_2 =\Delta y_3-\alpha y_2 && \text{in}\quad (0,T)\times  \Omega,\label{fre5}\\
 &y_3=-\Delta y_2+f(u^h)-f(\hat{u})-f'(\hat{u})\varphi_2&& \text{in}\quad (0,T)\times  \Omega,\label{fre6}\\
&\nabla y_2 \cdot \vec{n} =0 && \text{on } \quad (0,T)\times \partial\Omega, \label{fre7}\\
			&\nabla y_3\cdot \vec{n}=0
			&& \text{on }\quad (0,T)\times \partial\Omega,
			\label{fre8}
			\\
			&y_2(0,x)=0
			&& \text{in }\quad \Omega,\label{fre9}
\end{align}
where $p_{y_1}= p_{v}-p_{\varphi_1}$.
    \par Multiplying $y_1$ and $-\lambda \Delta y_2$ with \eqref{fre1} and \eqref{fre5}, respectively and integrating the resulting equalities over $\Omega$, we obtain
    \begin{align}
        &\frac{1}{2} \frac{d}{dt}|| y_1||^2_{L^2(\Omega)}+  \mu|| \nabla y_1||^2_{L^2(\Omega)}+ \frac{\lambda}{2}\frac{d}{dt}|| \nabla y_2||^2_{L^2(\Omega)}\nonumber\\
        =& -\lambda \int_{\Omega}(u \nabla w+ \hat{u} \nabla y_3+ y_2 \nabla \hat{w})\cdot y_1\,dx \nonumber\\& \hspace{1cm}- \lambda\int_{\Omega}( \Delta y_3-  v \cdot \nabla u-  y_1 \cdot \nabla \hat{u} -  \hat{v} \cdot \nabla y_2 - \alpha y_2)\Delta y_2dx \nonumber\\
        =&\,  \lambda \int_{\Omega}\Big( y_1 \cdot \nabla u \,w + y_1 \cdot \nabla y_2\, \hat{w}+ y_1 \cdot \nabla \hat{u}(-\Delta y_2+f(u^h)-f(\hat{u})-f'(\hat{u}))\Big)\,dx \nonumber\\
        & \hspace{1cm}-\lambda\int_{\Omega} ( \Delta y_3-  v \cdot \nabla u-  \hat{v} \cdot \nabla y_2 -\alpha y_2)\Delta y_2\,dx+  \lambda\int_{\Omega} y_1 \cdot \nabla \hat{u} \Delta y_2\,dx. \label{fre10}
    \end{align}
    Now, 
    \begin{align*}
         & \lambda \int_{\Omega}\Big( y_1 \cdot \nabla u \,w + y_1 \cdot \nabla y_2\, \hat{w}+ y_1 \cdot \nabla \hat{u}(f(u^h)-f(\hat{u})-f'(\hat{u}))\Big)\,dx\\
         = &\,  \lambda \int_{\Omega} \Big(y_1 \cdot \nabla u \,(-\Delta u + f(u^h)-f(\hat{u}))+ y_1 \cdot \nabla y_2\, \hat{w}+ y_1 \cdot \nabla \hat{u}\big(f(u^h)-f(\hat{u})-f'(\hat{u})\big)\Big)\,dx\\
         \leq &\, \lambda ||y_1||_{L^3(\Omega)}|| u||_{L^6(\Omega)}|| \nabla \Delta u||_{L^2(\Omega)}+ \lambda ||y_1||_{L^3(\Omega)}|| u||_{L^6(\Omega)} || \nabla (f(u^h)-f(\hat{u}))||_{L^2(\Omega)}\\
         &\,\hspace{2cm}+ \lambda || y_1||_{L^3(\Omega)}|| \hat{u}||_{L^6(\Omega)} ||\nabla(f(u^h)-f(\hat{u})-f'(\hat{u})\varphi_2) ||_{L^2(\Omega)}\\
         & \hspace{3cm}+ \lambda ||y_1||_{L^3(\Omega)} ||y_2||_{L^6(\Omega)} || \nabla \hat{w}||_{L^2(\Omega)}\\
         \leq &\,  \lambda ||y_1||_{L^3(\Omega)}|| \nabla u||_{L^2(\Omega)}|| \nabla \Delta u||_{L^2(\Omega)}+ \lambda ||y_1||_{L^3(\Omega)}|| \nabla u||_{L^2(\Omega)} || \nabla (f(u^h)-f(\hat{u}))||_{L^2(\Omega)}\\
         &\,\hspace{1cm}+  \lambda|| y_1||_{L^3(\Omega)}|| \hat{u}||_{L^6(\Omega)} ||\nabla(f(u^h)-f(\hat{u})-f'(\hat{u})\varphi_2) ||_{L^2(\Omega)} \\
         &\hspace{2cm}+  \lambda ||y_1||_{L^3(\Omega)} ||y_2||_{L^6(\Omega)} || \nabla \hat{w}||_{L^2(\Omega)}
    \end{align*}
    and 
    \begin{align*}
    &-\lambda\int_{\Omega} ( \Delta y_3-  v \cdot \nabla u-  \hat{v} \cdot \nabla y_2 -\alpha y_2)\Delta y_2\,dx\\
        =&\, -\lambda  \int_{\Omega}  \Big( \Delta \big(-\Delta y_2+f(u^h)-f(\hat{u})-f'(\hat{u})\varphi_2\big)-  v \cdot \nabla u-  \hat{v} \cdot \nabla y_2 -\alpha y_2 \Big)\Delta y_2\,dx\\
        \leq & \,-\lambda || \nabla \Delta y_2||^2_{L^2(\Omega)}+ || \nabla \big( f(u^h)-f(\hat{u})-f'(\hat{u})\varphi_2\big)||_{L^2(\Omega)}||\nabla \Delta y_2||_{L^2(\Omega)}
        \\&+ \lambda|| \hat{v}||_{L^3(\Omega)} || y_2||_{L^6(\Omega)} ||\nabla \Delta y_2||_{L^2(\Omega)}+  \lambda ||v||_{L^3(\Omega)} ||u||_{L^6(\Omega)} ||\nabla \Delta y_2||_{L^2(\Omega)}.
    \end{align*}

    Now,
    \begin{align*}
        &||\nabla(f(u^h)-f(\hat{u})-f'(\hat{u})\varphi_2) ||_{L^2(\Omega)}\\
        =&\, ||\nabla((u^h)^3- (\hat{u})^3-3(\hat{u}) \varphi_2 -y_2)||_{L^2(\Omega)}\\
        =&\, ||\nabla (u^3+ 3(u^h)^2\hat{u}-3 u^h(\hat{u})^2-3(\hat{u}) \varphi_2 -y_2)||_{L^2(\Omega)}\\
        =&\, ||\nabla(u^3+ 3u^2\hat{u}+ 3(\hat{u})^2 y_2 -y_2 )||_{L^2(\Omega)}\\
        \leq & \, 3||u||^2_{L^6(\Omega)}||\nabla u||_{L^6(\Omega)}+ 6||u||_{L^6(\Omega)}||\nabla u||_{L^6(\Omega)} ||\hat{u}||_{L^6(\Omega)}+3 ||u||^2_{L^6(\Omega)}||\nabla \hat{u}||_{L^6(\Omega)}\\
        &\hspace{1cm}+ 6 ||\hat{u}||_{L^6(\Omega)}|| \nabla \hat{u} ||_{L^6(\Omega)}||y_2||_{L^6(\Omega)}+3 ||\hat{u}||^2_{L^\infty(\Omega)}||\nabla y_2||_{L^2(\Omega)}+ ||\nabla y_2 ||_{L^2(\Omega)}.
    \end{align*}
   Hence \eqref{fre10} yields 
    \begin{align*}
        \frac{d}{dt}|| y_1||^2_{L^2(\Omega)}+  2(\mu-\lambda)|| \nabla y_1||^2_{L^2(\Omega)}+ \frac{d}{dt}\lambda || \nabla y_2||^2_{L^2(\Omega)} + \lambda ||\nabla \Delta y_2||^2_{L^2(\Omega)}\\
        \leq C_1(t)\Big(||y_1||^2_{L^2(\Omega)})+ \lambda ||\nabla y_2||^2_{L^2(\Omega)}\Big)+C_2(t),\\
        \end{align*}
        where
        \begin{align*}
           C_1(t)= C\Big(1+ ||\hat{v}||^2_{L^3(\Omega)}+ ||v||^2_{L^3(\Omega)}+|| \nabla \hat{w}||^2_{L^2(\Omega)}+||\hat{u}||^4_{L^6(\Omega)}||\nabla \hat{u}||^2_{L^6(\Omega)}\notag\\
           + ||\hat{u}||^2_{L^6(\Omega)}||\hat{u}||^4_{L^\infty}+||\hat{u}||^2_{L^6(\Omega)}||\nabla \hat{u}||^2_{L^6(\Omega)}+ ||\hat{u}||^4_{L^\infty(\Omega)}\Big )
        \end{align*}
        and 
        \begin{align*}
            C_2(t)= C\Big(|| \nabla u||^2_{L^2(\Omega)}||\nabla \Delta u||^2_{L^2(\Omega)}+||u||^2_{L^6(\Omega)}|| \nabla (f(u^h)-f(\hat{u}))||^2_{L^2(\Omega)}\\
+||\hat{u}||^2_{L^6(\Omega)}||u||^4_{L^6}||\nabla u||^2_{L^6(\Omega)}+ ||\hat{u}||^4_{L^6(\Omega)}||u||^2_{L^6(\Omega)}|| \nabla u||^2_{L^6(\Omega)}\\
            + ||\hat{u}||^2_{L^6(\Omega)}||u||^4_{L^6(\Omega)}||\nabla \hat{u}||^2_{L^6(\Omega)}+||u||^4_{L^6(\Omega)}||\nabla u||^2_{L^6(\Omega)}\\
            + ||u||^2_{L^6(\Omega)}||\nabla u||^2_{L^6(\Omega)}||\hat{u}||^2_{L^6(\Omega)}+||u||^4_{L^6(\Omega)}||\nabla \hat{u}||^2_{L^6(\Omega)}\Big).
        \end{align*}
        
       Now, Gronwall's inequality implies
       \begin{align*}
           || y_1(t)||^2_{L^2(\Omega)}+ \lambda || \nabla y_2(t)||^2_{L^2(\Omega)}+ \int_{0}^{T} \Big(2 (\mu-\lambda)|| \nabla y_1||^2_{L^2(\Omega)}+ \lambda ||\nabla \Delta y_2||^2_{L^2(\Omega)}\Big)\,dt\\
  \lesssim \Big(\int_{0}^{T} C_2(t) \,dt\Big)e^{\int_{0}^{T} C_1(t)\,dt}.
       \end{align*}
        Again,
        \begin{align*}
            \int_{0}^{T} || \partial_t y_1||_{\mathcal{V}^*} \,dt \leq \int_{0}^{T} \Big(  \mu ||\nabla y_1||_{L^2(\Omega)} + \lambda||u||_{L^3(\Omega)} || \nabla w||_{L^2(\Omega)}\\
            +\lambda||\hat{u}||_{L^3(\Omega)}|| \nabla y_3||_{L^2(\Omega)}
            + \lambda||y_2||_{L^3(\Omega)}||\nabla \hat{w}||_{L^3(\Omega)}\Big)\,dt
        \end{align*}
        and 
        \begin{align*}
            \int_{0}^{T} || \partial_t y_2||_{H^1(\Omega)^*}\,dt\leq  \int_{0}^{T}\Big( ||v||_{L^6(\Omega)}||  u||_{L^3(\Omega)}+  ||y_1||_{L^6(\Omega)}|| \hat{u}||_{L^3(\Omega)} \\
            + || \hat{v}||_{L^6(\Omega)} ||y_2||_{L^3(\Omega)}+
            ||\nabla y_3||_{L^2(\Omega)} + || y_2||_{L^2(\Omega)}\Big)\,dt.
        \end{align*}
        Hence, from the Theorem \ref{thm 3.3} and the explicit form of $C_1$ and $C_2$ implies that 
        \begin{align*}
            || y_1(t)||^2_{L^2(\Omega)}+  || \nabla y_2(t)||^2_{L^2(\Omega)}+ \int_{0}^{T} \Big(|| \nabla y_1||^2_{L^2(\Omega)}+ \lambda ||\nabla \Delta y_2||^2_{L^2(\Omega)}+ || \partial_t y_2||^2_{H^1(\Omega)^*}\\+ || \partial_t y_1||^2_{\mathcal{V}^*} 
            +||y_3||^2_{H^1(\Omega)}\Big)\,dt \leq C(t, ||h||_{L^2(0,T;\mathcal{H})})\,||h||^2_{L^2(0,T;\mathcal{H})},
        \end{align*}
        where $\displaystyle \lim_{||h|| \to 0} C(t,||h||)=0$. This implies that
        \begin{equation*}
            \lim_{||h|| \to 0} \frac{||\mathcal{S}(\hat{\theta}+h)-\mathcal{S}(\hat{\theta})-\mathcal{D}\mathcal{S}(\hat{\theta})(h)||_{\mathcal{F}}}{||h||}=0.
        \end{equation*}

\end{proof}

  In the subsequent analysis, we will establish the variational inequality that the optimal controls must fulfill. Since $J$ is a quadratic functional, we can find the Fr\'echet derivative of the $J(\theta)=(\mathcal{S}(\theta), \theta)$ by the chain rule. At $\hat{\theta} \in \mathcal{U}$
  \begin{equation*}
      \mathcal{D}J(\hat{\theta})= \mathcal{D}_{(v,u,w)}J(\mathcal{S}(\hat{\theta}), \hat{\theta}) \cdot \mathcal{D}\mathcal{S}(\hat{\theta})+ \mathcal{D}_\theta J(\mathcal{S}(\hat{\theta}), \hat{\theta}).
  \end{equation*}
  Now, using the convexity of $\mathcal{U}_{ad}$, we obtain that for any local minimizer $\hat{\theta} \in \mathcal{U}_{ad}$ of $J$ in $\mathcal{U}_{ad}$,
  \begin{align*}
      \mathcal{D}J(\hat{\theta})(\theta- \hat{\theta}) \geq 0 \,\, \text{for all $\theta \in \mathcal{U}_{ad}$}.
  \end{align*}
  We summarize these facts in the following result.
  \begin{thm}\label{thm 4.4}
    Assume that the assumptions (A1) -- (A4) hold true.  Let $\hat{\theta} \in \mathcal{U}_{ad}$ be an local optimal control for $\mathcal{(OP)}$ with the associated state $\mathcal{S}(\hat{\theta})= (\hat{v}, \hat{u}, \hat{w})$, then for any $\theta \in \mathcal{U}_{ad}$, we have 
\begin{align*}
    \int_{(0,T)\times  \Omega}(\hat{u}-u_d)\varphi_2\, \di(x,t)+ \int_{(0,T)\times  \Omega}(\hat{v}-v_d)\cdot \varphi_1 \,\di(x,t)\\+ \beta \int_{(0,T)\times  \Omega} (\theta- \hat{\theta})\cdot \hat{\theta}\,\di(x,t) \geq 0,
\end{align*}
    where $\mathcal{D}\mathcal{S}(\hat{\theta})(\theta-\hat{\theta})= (\varphi_1, \varphi_2, \varphi_3)$ is the unique weak solution to the linearized system \eqref{lin1} -- \eqref{lin6} for $h=\theta-\hat{\theta}$. 
  \end{thm}
  \subsection{ First--order optimality condition} We now proceed to establish the first-order optimality condition for the control problem $\mathcal{(OP)}$. To achieve this, we first derive the adjoint state and demonstrate the well-posedness of its weak solution. Subsequently, a straightforward variational analysis leads to the desired inequality.  \\

\begin{thm}
    Assume that the assumptions (A1) -- (A4) hold true. Let $\hat{\theta}$ be an local optimal control for $\mathcal{(OP)}$ with the associated state $(\hat{v},\hat{u},\hat{w})=\mathcal{S}(\hat{\theta})$, then the following adjoint state
\begin{subequations}
    \begin{align}
        &-\frac{\partial \g_1}{\partial t}-  \mu \Delta \g_1 + \nabla q= -\g_2\nabla \hat{u}+ (\hat{v}-v_d) && \text{in} \,\, (0,T) \times \Omega \label{adeq1}\\
       &\nabla \cdot \g_1  =0 && \text{in} \,\, (0,T) \times \Omega\label{adeq2}\\
        &\g_1(x,t)=0 && \text{on} \,\, (0,T) \times \partial \Omega\label{adeq3}\\
        &\g_1(x,T)=0 && \text{in} \,\, \Omega\label{adeq4}\\
        &-\frac{\partial \g_2}{\partial t}-  \hat{v} \cdot \nabla \g_2+\alpha \g_2+ \lambda \nabla \hat{w} \cdot \g_1+f'(\hat{u})\g_3\notag\\
        &\hspace{4cm}=(\hat{u}-u_d)+ \Delta \g_3 && \text{in} \,\, (0,T) \times \Omega\label{adeq5}\\
        &\g_3=-\Delta \g_2-\lambda \nabla \hat{u} \cdot \g_1  && \text{in} \,\, (0,T) \times \Omega\label{adeq6}\\
        &\nabla \g_2 \cdot \vec{\eta}=0,\,\nabla \g_3 \cdot \vec{\eta}=0 &&\text{on} \,\, (0,T) \times \partial \Omega\label{adeq7}\\
        &\g_2(x,T)=0&& \text{in} \,\, \Omega\label{adeq8}
    \end{align}
\end{subequations}
has a unique solution $(\g_1, \g_2, \g_3)$ in $L^2(0,T;\mathcal{V}) \cap H^1(0,T;\mathcal{V}^*) \times L^2(0,T;H^2(\Omega)) \cap H^1(0,T;H^2(\Omega)) \times L^2((0,T)\times  \Omega)$.
\end{thm}

\begin{proof}
The proof follows from a similar argument as in Theorem \ref{thm 3.2} using the Fadeo--Galerkin method. For the sake of simplicity, we omit the detailed construction of the approximation scheme and present only the a-priori estimates only. 

   \par  Multiplying \eqref{adeq1}, \eqref{adeq5} and \eqref{adeq6} by $\g_1$, $\g_2$ and $-\Delta \g_2+ f'(\hat{u})\g_2$, respectively and integrating the resulting equalities over $\Omega$ to obtain
\begin{align*}
    &\,-\frac{1}{2}\frac{d}{dt} ||\g_1||^2_{L^2(\Omega)}+ \mu || \nabla \g_1||^2_{L^2(\Omega)}-\frac{1}{2}\frac{d}{dt} ||\g_2||^2_{L^2(\Omega)}+ || \Delta \g_2||^2_{L^2(\Omega)}+ \alpha || \g_2||^2_{L^2(\Omega)}\\
    =&\,  \int_{\Omega}(\hat{v}-v_d)\g_1\,dx+  \int_{\Omega} (\hat{u}-u_d)\g_2\,dx- \int_{\Omega} \g_2 \nabla \hat{u}\cdot \g_1 \,dx - \lambda \int_{\Omega}\nabla \hat{w}\cdot \g_1 \g_2\, dx\\
    &\hspace{2cm}+ \lambda
     \int_{\Omega} f(\hat{u}) \g_1 \cdot \nabla \g_2\,dx+ \int_{\Omega} f'(\hat{u})  \g_2 \Delta \g_2\,dx-\lambda \int_{\Omega} \g_1 \cdot \nabla \hat{u} \Delta \g_2\,dx
     \end{align*}
     \begin{align*}
     =& \int_{\Omega}(\hat{v}-v_d)\g_1\,dx+  \int_{\Omega} (\hat{u}-u_d)\g_2\,dx + \int_{\Omega}\g_1 \cdot \nabla \g_2 (\hat{u}+\lambda \hat{w})\,dx +\int_{\Omega} f'(\hat{u})  \g_2 \Delta \g_2\,dx\\
     & \hspace{5cm}+\lambda
     \int_{\Omega} f(\hat{u}) \g_1 \cdot \nabla \g_2\,dx -\lambda \int_{\Omega} \g_1 \cdot \nabla \hat{u} \Delta \g_2\,dx\\
     =& \int_{\Omega}(\hat{v}-v_d)\g_1\,dx+  \int_{\Omega} (\hat{u}-u_d)\g_2\,dx + \int_{\Omega}\g_1 \cdot \nabla \g_2 ( \hat{u}-\lambda \Delta \hat{u}) + \int_{\Omega} f'(\hat{u})  \g_2 \Delta \g_2\,dx\\
     & \hspace{9cm}-\lambda \int_{\Omega} \g_1 \cdot \nabla \hat{u} \Delta \g_2\,dx\\
     \leq & \,||\hat{v}-v_d||_{L^2(\Omega)}||\g_1||_{L^2(\Omega)}+  || f'(\hat{u})||_{L^\infty(\Omega)}|| \g_2||_{L^2(\Omega)}|| \Delta \g_2||_{L^2(\Omega)}\\
     &\hspace{0.5cm}+||\hat{u}-u_d||_{L^2(\Omega)}||\g_2||_{L^2(\Omega)}+ || \g_1||_{L^2(\Omega)} || \nabla \g_2||_{L^4(\Omega)} ||\hat{u}-\lambda \Delta \hat{u}||_{L^4(\Omega)} \\
     &\hspace{1cm}+ \lambda||\nabla  \g_1||_{L^2(\Omega)} ||\nabla \hat{u}||_{L^4(\Omega)}|| \nabla \g_2||_{L^4(\Omega)} +||\g_1||_{L^2(\Omega)}||\hat{u}||_{H^3(\Omega)}|| || \nabla \g_2||_{L^4(\Omega)} .                                          \end{align*}
Hence,
\begin{align*}
    -&\,\frac{1}{2}\frac{d}{dt} ||\g_1||^2_{L^2(\Omega)}+ (\mu-\lambda) || \nabla \g_1||^2_{L^2(\Omega)}-\frac{1}{2}\frac{d}{dt} ||\g_2||^2_{L^2(\Omega)}+\frac{1}{2} || \Delta \g_2||^2_{L^2(\Omega)}\\
    \leq&\, C_1\Big(||\g_1||^2_{L^2(\Omega)} + ||\g_2||^2_{L^2(\Omega)}\Big)+ C||\hat{u}-u_d||^2_{L^2(\Omega)}+ C || \hat{v}-v_d||^2_{L^2(\Omega)},
\end{align*}
where
\begin{align*}
    C_1=C(1+ ||\hat{u}-\Delta \hat{u}||^2_{L^4(\Omega)}+||\hat{u}||^2_{H^3(\Omega)}+||f'(\hat{u})||^2_{L^\infty(\Omega)}+||\hat{u}||^{16}_{L^4(\Omega)}).
\end{align*}
Choose $\mu > \lambda$, then integrating the above inequality over $(t,T)$ and applying Gronwall's inequality, we obtain
\begin{align*}
||\g_1(t)||^2_{L^2(\Omega)}+||\g_2(t)||^2_{L^2(\Omega)}+ \int_{0}^{t}\Big(2 (\mu-\lambda) || \nabla \g_1||^2_{L^2(\Omega)}+ || \Delta \g_2||^2_{L^2(\Omega)}\Big)\,dt \\
 \leq C(||v_0||_{L^2(\Omega)^n}, ||\phi_0||_{H^1(\Omega)}, ||\theta||_{\mathcal{H}},T).
\end{align*}

Furthermore,
\begin{align*}
    \int_{0}^{T} || \partial_t \g_1||^2_{\mathcal{V}^*}\,dt \leq &   2 \int_{0}^{T}\Big(  \mu^2 || \nabla \g_1||^2_{L^2(\Omega)}+  || \g_2||^2_{L^6(\Omega)} || \nabla \hat{u}||^2_{L^2(\Omega)}+ ||\hat{v}-v_d||^2_{L^2(\Omega)} \Big)\,dt
\end{align*}
and
\begin{align*}
     \int_{0}^{T} || \partial_t \g_2||^2_{H^2(\Omega)^*}\,dt \leq & 2 \int_{0}^{T}\Big(  ||\hat{v}||^2_{L^6(\Omega)}||\nabla \g_2||^2_{L^2(\Omega)}+  \lambda^2 ||\nabla \hat{w}||^2_{L^2(\Omega)} || \g_1||^2_{L^3(\Omega)} \\
     &+ ||f'(\hat{u})||^2_{L^3(\Omega)}||\g_3||^2_{L^2(\Omega)})+ ||\hat{u}-u_d||^2_{L^2(\Omega)}+ ||\g_3||^2_{L^2(\Omega)}\Big) \, dt.
\end{align*}
\par The uniqueness of weak solutions follows from the linearity of the system. By employing similar techniques as in Theorem \ref{thm 3.2}, one can derive the estimate for $q$. 

\end{proof}

  \begin{lemma}
      Let the conditions (A1) -- (A4) are satisfied and $\hat{\theta} \in \mathcal{U}_{ad}$ be an optimal control for the problem $\mathcal{(OP)}$, with the associated state $\mathcal{S}(\hat{\theta})=(\hat{v},\hat{u},\hat{w})$. Let $(\g_1, \g_2, \g_3)$ denote the solution of the corresponding adjoint problem \eqref{adeq1} -- \eqref{adeq8}. Then the optimal control $\hat{\theta}$ satisfies the following variational inequality:
      \begin{equation*}
           \int_{(0,T)\times  \Omega} (\g_1+ \beta \hat{\theta})\cdot (\theta-\hat{\theta})\,\di(x,t) \geq 0\,\,\text{ for all $\theta \in \mathcal{U}_{ad}$}.
      \end{equation*}
  \end{lemma}
\begin{proof}
    Choose the test functions $\g_1, \, \g_2$ and $\g_3$ for the weak formulations of \eqref{lin1}, \eqref{lin2} and \eqref{lin3}, respectively, we obtain
\begin{align}
    \int_{0}^{T}\langle \partial_t \varphi_1, \g_1 \rangle\,dt + \mu \int_{(0,T)\times  \Omega} \nabla \varphi_1 \cdot \nabla \g_1\, \di(x,t) + \lambda \int_{(0,T)\times  \Omega} \varphi_2  \nabla \hat{w} \cdot \g_1\,\di(x,t)\nonumber\\+ \lambda \int_{(0,T)\times  \Omega} \hat{u}\nabla \varphi_3 \cdot \g_1\,\di(x,t)
    +\int_{(0,T)\times  \Omega} \partial_t \varphi_2 \g_2 \,\di(x,t) + \int_{(0,T)\times  \Omega} \nabla \varphi_3 \cdot \nabla \g_2 \,\di(x,t)\nonumber\\
    +\int_{(0,T)\times  \Omega} \varphi_1 \cdot \nabla \hat{u} \g_2 \,\di(x,t)+ \int_{(0,T)\times  \Omega} \hat{v} \cdot \nabla \varphi_2 \g_2 \,\di(x,t)+ \alpha \int_{(0,T)\times  \Omega} \varphi_2 \g_2 \,\di(x,t)\nonumber\\
    - \int_{(0,T)\times  \Omega}\varphi_3 \g_3 \,\di(x,t)+ \int_{(0,T)\times  \Omega} \nabla \varphi_2 \cdot \nabla \g_3 \,\di(x,t)+ \int_{(0,T)\times  \Omega} f'(\hat{u})\phi_1\g_3 \,\di(x,t)\nonumber\\
    =\int_{(0,T)\times  \Omega}(\theta-\hat{\theta}) \cdot \g_1 \,\di(x,t), \label{OCP1}
\end{align}
where $h= \theta-\hat{\theta}$ and $(\g_1,\g_2,\g_3)$ is the weak solution of the adjoint state \eqref{adeq1} - \eqref{adeq8}. \par Again, choose the test functions $\varphi_1$, $\varphi_2$ and $\varphi_3$ for the weak formulations of  \eqref{adeq1}, \eqref{adeq5} and \eqref{adeq6}, respectively, we get
\begin{align}
    \int_{0}^{T}\langle-\partial_t \g_1, \varphi_1\ \rangle\,dt+ \mu \int_{(0,T)\times  \Omega} \nabla \varphi_1 \cdot \nabla \g_1\, \di(x,t)+ \int_{(0,T)\times  \Omega} \g_2  \nabla \hat{u} \cdot \varphi_1 \, \di(x,t)\nonumber\\
    - \int_{(0,T)\times  \Omega} (\hat{v}-v_d)\cdot \varphi_1\, \di(x,t)
    -\int_{(0,T)\times  \Omega} \partial_t \g_2 \varphi_2 \, \di(x,t)- \int_{(0,T)\times  \Omega} \hat{v}\cdot \nabla g_2 \varphi_2 \, \di(x,t)\nonumber\\ 
    + \alpha \int_{(0,T)\times  \Omega} \g_2 \varphi_2  \, \di(x,t)+ \lambda \int_{(0,T)\times  \Omega} \nabla \hat{w} \g_1 \varphi_2 \, \di(x,t)
    + \int_{(0,T)\times  \Omega}f'(\hat{u}) \g_3 \varphi_2 \, \di(x,t)\nonumber\\
    - \int_{(0,T)\times  \Omega} (\hat{u}- u_d) \varphi_2 \, \di(x,t)+ \int_{(0,T)\times  \Omega} \nabla \g_3 \cdot \nabla \varphi_2 \, \di(x,t)-\int_{(0,T)\times  \Omega} \g_3 \varphi_3 \, \di(x,t)\nonumber\\
    + \int_{(0,T)\times  \Omega} \nabla \g_2 \cdot \nabla \varphi_3 \, \di(x,t)
    - \lambda \int_{(0,T)\times  \Omega} \nabla \hat{u} \cdot \g_1 \varphi_3 \, \di(x,t)=0, \label{OCP2}
\end{align}
where $(\varphi_1, \varphi_2, \varphi_3)$ is the solution of the linearized system \eqref{lin1} -- \eqref{lin6}.\\

Hence from \eqref{OCP1} -- \eqref{OCP2}, it implies that
\begin{align*}
    \int_{(0,T)\times  \Omega}(\hat{u}-u_d)\varphi_2\, \di(x,t)+ \int_{(0,T)\times  \Omega}(\hat{v}-v_d)\cdot \varphi_1 \,\di(x,t)\\
    = \int_{(0,T)\times  \Omega}(\theta-\hat{\theta}) \cdot \g_1 \,\di(x,t).
\end{align*}
Therefore from Theorem \ref{thm 4.4}, it follows that
\begin{align}
    \int_{(0,T)\times  \Omega} (\g_1+ \beta \hat{\theta})\cdot (\theta-\hat{\theta})\,\di(x,t) \geq 0. \label{fop}
\end{align}

\end{proof}

\begin{remark}
Moreover, $\mathcal{U}_{ad}$ is a nonempty, closed and convex subset of $L^2(0,T;\mathcal{H})$, then the variational inequality \eqref{fop} implies that the optimal control can be characterized pointwise as the $L^2(0,T;\mathcal{H})$ orthogonal projection of $-\frac{1}{\beta}\g_1$ onto $\mathcal{U}_{ad}$. This provides an explicit representation of the optimal control with respect to the adjoint variable.    
\end{remark}

\section*{Future scope}As a possible direction for future research, the present optimal control problem may be extended to porous media flow models, since the Stokes--Cahn--Hilliard--Oono system naturally arises in the study of multiphase flows through porous structures. In particular, it would be interesting to investigate the multiscale analysis of the associated optimal control problem in periodically perforated domains.
\begin{acknowledgements}
 AK expresses gratitude to NBHM (Ref. No.- 020 3/11/2019-R$\&$D-11/9247) for providing fellowship support during his doctoral studies..
 \end{acknowledgements}


\bibliographystyle{1}
\bibliography{cite}








\end{document}